\documentclass[11pt]{article} 

\usepackage[utf8]{inputenc} 

\usepackage{geometry} 
\usepackage{graphicx} 
\usepackage[english]{babel}

\usepackage{booktabs} 
\usepackage{array} 
\usepackage{paralist} 
\usepackage{verbatim} 
\usepackage[format=hang,singlelinecheck=false]{subfig} 
\usepackage{amsfonts}
\usepackage{amssymb}
\usepackage{amsmath}
\usepackage{amsthm}

\usepackage{epigraph}
\usepackage{bbm}

\usepackage{hyperref}
\hypersetup{
    colorlinks=true,
    citecolor=blue
}

\renewcommand{\le}{\leqslant}
\renewcommand{\ge}{\geqslant}

\newcommand{\id}{\mathbbmss 1}

\newcommand {\matl}{\left[ \begin{matrix}}
\newcommand {\matr}{\end{matrix}\right]}
\newcommand {\Exp}{ \mathbb E }
\newcommand {\sg}{\operatorname{sgn}}
\renewcommand {\Pr}{ \mathbb P }

\newcommand {\Var}{\mathbf {Var}}

\newcommand{\e}{\mathrm e}

\renewcommand{\d}{\mathrm{d}}

\newcommand{\cP}{\mathcal{P}}

\newcommand{\Otilde}{\Tilde{\mathcal{O}}}
\newcommand{\Otheta}{\Tilde{\Theta}}

\newcommand{\htheta}{\widehat{\theta}}

\newcommand{\hmu}{\widehat{\mu}}

\newcommand{\CI}{\operatorname{CI}}
\newcommand{\CIDS}{\operatorname{CI}^{\mathsf{DS}}}
\newcommand{\CIDShet}{\operatorname{CI}^{\mathsf{DS'}}}

\newcommand{\CIch}{\CI^{\textsf{Cheb}}}
\newcommand{\CIcher}{\CI^{\textsf{Cher}}}
\newcommand{\aCISN}{\operatorname{aCI}^{\mathsf{SN}}}
\newcommand{\aCISNp}{\operatorname{aCI}^{\mathsf{SN}+}}
\newcommand{\aCISNn}{\operatorname{aCI}^{\mathsf{SN}-}}
\newcommand{\aCISNphet}{\operatorname{aCI}^{\mathsf{SN}+'}}
\newcommand{\aCISNnhet}{\operatorname{aCI}^{\mathsf{SN}-'}}
\newcommand{\CISN}{\operatorname{CI}^{\mathsf{SN}}}
\newcommand{\CISNhet}{\operatorname{CI}^{\mathsf{SN}'}}
\newcommand{\LSN}{\operatorname{L}^{\mathsf{SN}}}
\newcommand{\MSN}{\operatorname{M}^{\mathsf{SN}}}
\newcommand{\USN}{\operatorname{U}^{\mathsf{SN}}}
\newcommand{\CICT}{\operatorname{CI}^{\mathsf{C}}}

\newcommand{\CICTplus}{\operatorname{CI}^{\mathsf{C+}}}

\newcommand{\CIstitch}{\operatorname{CI}^{\mathsf{Cstch}}}

\newcommand{\CR}{\operatorname{CR}}

\newcommand{\Mcat}{M^{\mathsf{C}}}
\newcommand{\Ncat}{N^{\mathsf{C}}}
\newcommand{\Msn}{M^{\mathsf{SN}}}

\newcommand{\lw}{\operatorname{lw}}
\newcommand{\up}{\operatorname{up}}

\DeclareMathOperator*{\Prw}{\Pr}

\usepackage{natbib}
\usepackage{wrapfig}
\usepackage{subfig}
\usepackage[capitalize,noabbrev]{cleveref}
\usepackage{authblk}
\usepackage{changepage}

\title{Catoni-style confidence sequences\\for heavy-tailed mean estimation\thanks{Published in \emph{Stochastic Processes and their Applications}, \url{https://doi.org/10.1016/j.spa.2023.05.007}.}}
\author[1]{Hongjian Wang}
\author[2]{Aaditya Ramdas}
\affil[1, 2]{Machine Learning Department, Carnegie Mellon University} 
\affil[2]{Department of Statistics and Data Science, Carnegie Mellon University}
\affil[ ]{\texttt{ \{hjnwang,aramdas\}@cmu.edu  }}

\date{June 13, 2023} 

\newtheorem{theorem}{Theorem}
\newtheorem{definition}{Definition}

\newtheorem{lemma}[theorem]{Lemma}
\newtheorem{proposition}[theorem]{Proposition}
\newtheorem{corollary}{Corollary}[theorem]

\newtheorem{remark}{Remark}
\newtheorem{assumption}{Assumption}

\DeclareMathOperator*{\polylog}{polylog}

\begin{document}

\maketitle

\begin{abstract}
A confidence sequence (CS) is a sequence of confidence intervals that is valid at arbitrary data-dependent stopping times. These are useful in applications like A/B testing, multi-armed bandits, off-policy evaluation, election auditing, etc. We present three approaches to constructing a confidence sequence for the population mean, under the minimal assumption that only an upper bound $\sigma^2$ on the variance is known. While previous works rely on light-tail assumptions like boundedness or subGaussianity (under which all moments of a distribution exist), the confidence sequences in our work are able to handle data from a wide range of heavy-tailed distributions. The best among our three methods --- the Catoni-style confidence sequence --- performs remarkably well in practice, essentially matching the state-of-the-art methods for $\sigma^2$-subGaussian data, and provably attains the $\sqrt{\log \log t/t}$ lower bound due to the law of the iterated logarithm.
Our findings have important implications for sequential experimentation with unbounded observations, since the $\sigma^2$-bounded-variance assumption is more realistic and easier to verify than $\sigma^2$-subGaussianity (which implies the former). We also extend our methods to data with infinite variance, but having $p$\textsuperscript{th} central moment ($1<p<2$).
\end{abstract}

\section{Introduction}\label{sec:intro}

We consider the classical problem of sequential nonparametric mean estimation. As a motivating example, let $P$ be a distribution on $\mathbb R$ from which a stream of i.i.d.\ sample $X_1, X_2, \dots$ is drawn. The mean of the distribution,
\begin{equation}
    \mu = \int_{\mathbb R} x \, \d P = \Exp[X_i],
\end{equation}
is unknown and is our estimand. The traditional and most commonly studied approaches to this problem include, among others, the construction of \emph{confidence intervals} (CI). That is, we construct a $\sigma(X_1, \dots, X_t)$-measurable random interval $\CI_t$ for each $t \in \mathbb N^+$ such that
\begin{equation}\label{eqn:ci}
   \forall t \in \mathbb N^+, \quad \Pr[\mu \in \CI_t] \ge 1 - \alpha.
\end{equation}
It is, however, also well known that confidence intervals suffer from numerous deficiencies. For example, random stopping rules frequently arise in sequential testing problems, and it is well known that confidence intervals \eqref{eqn:ci} typically \emph{fail} to satisfy the guarantee
\begin{equation}\label{eqn:stopped-coverage}
    \forall \, \text{stopping time } \tau, \quad \Pr[\mu \in \CI_\tau] \ge 1 - \alpha.
\end{equation}
In other words, traditional confidence intervals are invalid and may undercover at stopping times.
To remedy this, a switch of order between the universal quantification over $t$ and the probability bound in the definition of CI \eqref{eqn:ci} was introduced \citep{darling1967confidence}:
\begin{equation}\label{eqn:cs}
    \Pr[\forall t \in \mathbb N^+, \; \mu \in \CI_t] \ge 1 - \alpha.
\end{equation}
The random intervals\footnote{We use ``confidence interval" to also refer to confidence sets, allowing $\CI_t$ sometimes to be the finite union of intervals.} $\{ \CI_t \}$ that satisfy the property above are called a $(1-\alpha)$-\emph{confidence sequence} (CS). The definition of CS \eqref{eqn:cs} and the property of stopped coverage \eqref{eqn:stopped-coverage} are actually equivalent, due to \citet[Lemma 3]{howard2021time}.

It is known that CSs do not suffer from the perils of applying CIs in sequential settings (e.g.\ continuous monitoring or peeking at CIs as they arrive). For example, \citet[Figure 1(b)]{howard2021time} shows in a similar context that the cumulative type-I error grows without bound if a traditional confidence interval is continuously monitored, but a confidence sequence has the same error bounded at $\alpha$; also see \citet{johari2017peeking} for a similar phenomenon stated in terms of p-values.

Prior studies on constructing confidence sequences for mean $\mu$ hinge on certain stringent assumptions on $P$.  \citet{darling1967confidence} considered exclusively the case where $P$ was a normal distribution. \citet{jennison1989interim} made the same parametric assumption. Later authors including \cite{lai1976confidence,csenki1979note}, and recently \citet{johari2015always} (who notably defined ``always-valid p-values'') allowed $P$ to be a distribution belonging to a fixed exponential family. More recently, \citet{howard2020time,howard2021time} performed a systematic study of nonparametric confidence sequences, whose assumptions on $P$ ranged among subGaussian, sub-Bernoulli, sub-gamma, sub-Poisson, and sub-exponential, which in most cases considered involve a bounded moment generating function, and in particular, that all moments exist. The latest advance in CSs was the paper by \citet{waudby2020estimating} that studied the case of bounded $P$ largely because of its ``betting" set-up, which of course implies all moments exist. Finally, the prior result closest to our setting is
a recent study on heavy-tailed bandits by \citet[Proposition 5]{agrawal2021regret}, whose implicit CS is based on empirical likelihood techniques, but it demands a nontrivial and costly optimization computation and its code is currently not publicly available.

In this paper, we remove all the parametric and tail lightness assumptions of the existing literature mentioned above, and make instead only one simple assumption (\cref{ass:var} in \cref{sec:setup}): the variance of the distribution exists and is upper bounded by a constant $\sigma^2$ known \emph{a priori},
\begin{equation}\label{eqn:var-bound}
    \int_{\mathbb R}(x - \mu)^2 \, \d P = \Var[X_i] \le \sigma^2.
\end{equation}

Further, we shall show that, even under this simple assumption that allows for a copious family of heavy-tailed distributions (whose third moment may be infinite), the $(1-\alpha)$-CS $\{ \CI_t \}$ which we shall present achieves remarkable width control. We characterize the tightness of a confidence sequence from two perspectives. First, the rate of \emph{shrinkage}, that is how quickly $|\CI_t|$, the width of the interval $\CI_t$, decreases as $t \to \infty$; Second, the rate of \emph{growth}, that is how quickly $|\CI_t|$ increases as $\alpha \to 0$. It is useful to review here how the previous CIs and CSs in the literature behave in these regards. Chebyshev's inequality, which yields~\eqref{eqn:ci} when requiring \eqref{eqn:var-bound}, states that
\begin{equation}\label{eqn:chebci}
    \CIch_t = \left[\widehat{\mu}_t \pm \frac{\sigma}{\sqrt{\alpha t}}\right], \quad\text{where } \widehat{\mu}_t = \frac{\sum_{i=1}^t X_i}{t}
\end{equation}
forms a $(1-\alpha)$-CI at every $t$, which is of shrinkage rate $\mathcal{O}(t^{-1/2})$ and growth rate $\mathcal{O}(\alpha^{-1/2})$. Strengthening the assumption from \eqref{eqn:var-bound} to subGaussianity with variance factor $\sigma^2$  \citep[Section 2.3]{boucheron2013concentration}, the Chernoff bound ensures that $(1-\alpha)$-CIs can be constructed by
\begin{equation}\label{eqn:chernoff}
    \CIcher_t = \left[\widehat{\mu}_t \pm \sigma \sqrt{\frac{2\log(2/\alpha)}{t}} \right],
\end{equation}
i.e.\ the stronger subGaussianity assumption leads to a sharper growth rate of $\mathcal{O}\left(\sqrt{\log(1/\alpha)}\right)$. It is \citet[Proposition 2.4]{catoni2012challenging} who shows the striking fact that by discarding the empirical mean $\widehat{\mu}_t$ and using an influence function instead to stabilize the outliers associated with heavy-tailed distributions, a $\mathcal{O}(\log(1/\alpha))$ growth rate can be achieved even when only the variance is bounded  \eqref{eqn:var-bound}; similar results can be found in the recent survey by \citet{lugosi2019mean}.

In the realm of confidence sequences, we see that recent results by \citet{howard2021time, waudby2020estimating}, while often requiring stringent Chernoff-type assumption on the distribution, all have $\mathcal{O}\left(\sqrt{\log t/t}\right)$ shrinkage rates and $\mathcal{O}\left(\sqrt{\log(1/\alpha)}\right)$ growth rates. For example, Robbins' famous two-sided normal mixture confidence sequence for subGaussian $\cP$ with variance factor $\sigma^2$ (see e.g., \citet[Equation (3.7)]{howard2021time}) is of the form
\begin{equation}\label{eqn:nmix-cs}
    \CI_t^{\mathsf{NMix}} = \left[\widehat{\mu}_t \pm \frac{\sigma\sqrt{(t+1) \log \frac{4(t+1)}{\alpha^2}}}{t} \right].
\end{equation}

The best among the three confidence sequences in this paper (\cref{thm:cs-cat}) draws direct inspiration from \citet{catoni2012challenging}, and achieves a provable shrinkage rate of $\Otilde(t^{-1/2})$, where the $\Otilde$ hides $\polylog t$ factors, and growth rate $\mathcal{O}(\log(1/\alpha))$. A fine-tuning of it leads to the exact shrinkage rate $\mathcal{O}(\sqrt{\log \log t/t})$, matching the lower bound of the law of the iterated logarithm under precisely the same assumption \eqref{eqn:var-bound}. The significance of this result, in conjunction with \citet{howard2021time}, is that moving from one-time valid interval estimation to anytime valid interval\emph{s} estimation (confidence sequences), no significant excess width is necessary to be incurred; nor does weakening the distribution assumption from sub-exponential to finite variance results in any cleavage of interval tightness, in both CI and CS alike. Our experiments demonstrate that published subGaussian CSs are extremely similar to our finite-variance CSs, but the former assumption is harder to check and less likely to hold (all moments may not exist for unbounded data in practice). We summarize and compare the mentioned works in terms of tightness in \cref{table:comparison-tightness}.

\begin{table}[h]
\vskip 0.15in
\begin{center}
\begin{scriptsize}
\begin{tabular}{p{3cm}|p{3.45cm}|p{5cm}}
\toprule
 & CI & CS  \\
\hline
Light-tailed \newline (MGF exists)& Chernoff bound (EM): \newline $\mathcal{O}(t^{-1/2})$, $\mathcal{O}(\sqrt{\log(1/\alpha)})$ & \citet{howard2021time} (EM): \newline $\mathcal{O}(\sqrt{\log t / t})$ or $\mathcal{O}(\sqrt{\log \log  t / t})$,
\newline
$\mathcal{O}(\sqrt{\log(1/\alpha)})$ 
\\
\hline
\vspace{1em} Heavy-tailed \newline (only finite variance)& Chebyshev inequality (EM): \newline $\mathcal{O}(t^{-1/2})$, $\mathcal{O}(\alpha^{-1/2})$. \newline \newline \citet{catoni2012challenging}: \newline  $\mathcal{O}(t^{-1/2})$, $\mathcal{O}(\log(1/\alpha))$ & This paper (\cref{cor:catoni-tight}): \newline $\mathcal{O}(\log \log t \sqrt{\log t / t}))$, $\mathcal{O}(\log(1/\alpha))$  w.h.p. 
\newline \newline This paper (\cref{cor:loglogt}): \newline $\mathcal{O}(\sqrt{\log\log t/t})$, $\mathcal{O}(\sqrt{\log(1/\alpha)})$ w.h.p.
\\
\hline 
\vspace{1em} Heavier-tailed \newline (only finite $p$\textsuperscript{th} moment) & Markov inequality (EM): \newline $\mathcal{O}(t^{-(p-1)/{p}})$, $\mathcal{O}(\alpha^{-1/p})$ \newline \newline \citet{chen2021generalized}: \newline $\mathcal{O}(t^{-(p-1)/p})$, $\mathcal{O}(\log(1/\alpha))$ & \vspace{1em} This paper (\cref{cor:infvar-tight}): \newline $\mathcal{O}(t^{-(p-1)/p}\log t)$, $\mathcal{O}(\log(1/\alpha))$  w.h.p.
\\
\bottomrule
\end{tabular}
\end{scriptsize}
\caption{Comparison of asymptotic tightnesses among prominent confidence intervals and confidence sequences.
Here ``(EM)'' indicates that the corresponding CI or CS is constructed around the empirical mean; ``w.h.p.'' stands for ``with high probability'' (used when the interval widths are not deterministic). 
The ``Markov inequality" bounds in the last cell of the ``CI" column can be derived from e.g.\ the martingale $L^p$ bound Lemma 7 in \citet{wang2021convergence}, Appendix A.
}
\label{table:comparison-tightness}
\end{center}
\vskip -0.1in
\end{table}


\section{Problem set-up and notations}\label{sec:setup}

Let $\{ X_t\}_{ t \in \mathbb N^+ }$ be a real-valued stochastic process adapted to the filtration $\{ \mathcal{F}_t \}_{ t \in \mathbb N_0}$ where $\mathcal{F}_0$ is the trivial $\sigma$-algebra. We make the following assumptions.
\begin{assumption}\label{ass:mean}
    The process has a constant, unknown conditional expected value:
    \begin{equation}
        \forall t\in\mathbb N^+, \quad \Exp[X_t \mid \mathcal{F}_{t-1}] = \mu.
    \end{equation}
\end{assumption}
\begin{assumption}\label{ass:var}
    The process is conditionally square-integrable with a uniform upper bound, known \emph{a priori}, on the conditional variance:
    \begin{equation}
        \forall t\in\mathbb N^+, \quad \Exp[(X_t -\mu)^2  \mid \mathcal{F}_{t-1}] \le \sigma^2.
    \end{equation}
\end{assumption}


The task of this paper is to construct confidence sequences $\{ \CI_t \}$ for $\mu$ from the observations $X_1, X_2, \dots$, that is,
\begin{equation}
    \Pr[\forall t \in \mathbb N^+, \ \mu \in \CI_t] \ge 1 - \alpha.
\end{equation}

We remark that our assumptions, apart from incorporating the i.i.d.\ case (with $\Exp[X_t] = \mu$ and $\Var[X_t] \le \sigma^2$) mentioned in \cref{sec:intro}, allow for a wide range of settings. The \cref{ass:mean} is equivalent to stating that the sequence $\{ X_t - \mu \}$ forms a \emph{martingale difference} (viz. $\{ \sum_{i=1}^t (X_i - \mu) \}$ is a martingale), which oftentimes arises as the model for non-i.i.d., state-dependent noise in the optimization, control, and finance literature (see e.g.\ \citet{kushner2003stochastic}). A very simple example would be the drift estimation setting with the stochastic differential equation $\d G_t = \sigma f(G_t, t) \d W_t + \mu \d t$, where $f$ is a function such that $|f(G_t, t)| \le 1$ and $W_t$ denotes the standard Wiener process. When sampling $X_t = G_t - G_{t-1}$, the resulting process $\{ X_t \}$ satisfies our \cref{ass:mean} and \cref{ass:var}. 

We further note that \cref{ass:mean} and \cref{ass:var} can be weakened to drifting conditional means and growing conditional $p$\textsuperscript{th} central moment bound ($1 < p \le 2$) respectively, indicating that our framework may encompass \emph{any} $L^p$ stochastic process $\{X_t\}$. These issues are to be addressed in \cref{sec:inf-var-in-text} and \cref{sec:non-constant}, while we follow \cref{ass:mean} and \cref{ass:var} in our exposition for the sake of simplicity. Finally, we remark that the requirement for a known moment upper bound like \cref{ass:var} may seem restrictive, but is known to be \emph{minimal} in the sense that no inference on $\mu$ would be possible in its absence, which we shall discuss in \cref{sec:minimality}. 

Throughout the paper, an auxiliary process $\{\lambda_t\}_{t \in \mathbb N^+}$ consisting of predictable coefficients (i.e.\ each $\lambda_t$ is an $\mathcal{F}_{t-1}$-measurable random variable) is used to fine-tune the intervals. 
We denote by $[m \pm w]$ the open or closed (oftentimes the endpoints do not matter) interval $[m-w, m+w]$ or $(m-w, m+w)$ to simplify the lengthy expressions for CIs and CSs; and by $\min(I)$, $\max(I)$ respectively the lower and upper endpoints of an interval $I$. The asymptotic notations follows the conventional use: for two sequences of nonnegative numbers $\{a_t\}$ and $\{b_t\}$, we write $a_t = \mathcal O(b_t)$ if $\limsup_{t\to\infty} a_t/b_t < \infty$, $a_t = \Theta(b_t)$ if both $a_t = \mathcal O(b_t)$ and $b_t = \mathcal O(a_t)$ hold, and $a_t \asymp b_t$ if $\lim_{t\to\infty} a_t/b_t$ exists and $0 < \lim_{t\to\infty} a_t/b_t < \infty$. We write $a_t = \polylog(b_t)$ if there exists a universal polynomial $p$ such that $a_t = \mathcal{O}(p(\log b_t))$. Finally, if $a_t = \mathcal{O}(b_t  \polylog(t))$, we say $a_t = \tilde{\mathcal{O}}(b_t)$.

\section{Confidence sequence 
via
the Dubins-Savage inequality}

The following inequality by \citet{dubins1965tchebycheff} is widely acknowledged to be a seminal result in the martingale literature and it will be the foundation of our first confidence sequence.

\begin{lemma}[Dubins-Savage inequality]\label{lem:ds} Let $\{M_t\}$ be a square-integrable martingale with $M_0 = 0$ and $V_t = \Exp[ (M_t - M_{t-1}) ^2 \mid \mathcal F_{t-1} ]$. Then, for all $a, b > 0$,
\begin{equation}
    \Pr\left[ \exists t\in \mathbb N^+, \; M_t \ge a + b \sum_{i=1}^t V_i \right] \le \frac{1}{ab + 1}.
\end{equation}
\end{lemma}
We prove \cref{lem:ds} in \cref{sec:ds-discussion} for completeness.
Recall from \cref{sec:setup} that $\{ \lambda_t \}_{t \in \mathbb N ^+}$ is a sequence of predictable coefficients. Define processes
\begin{equation}\label{eqn:ds-martingale}
    M^+_t = \sum_{i=1}^t \lambda_i (X_i -\mu), \quad  M^-_t = \sum_{i=1}^t -\lambda_i (X_i -\mu).
\end{equation}
As a consequence of \cref{ass:mean}, $\Exp[\lambda_t (X_t -\mu) \mid \mathcal{F}_{t-1}] =0$. Hence both of $\{ M^+_t\}$ and $\{ M^-_t\}$ are martingales. Applying \cref{lem:ds} to these two martingales yields the following result.
\begin{theorem}[Dubins-Savage confidence sequence]\label{thm:cids} Let $\{ \lambda_t \}_{t \in \mathbb N ^+}$ be any predictable process. The following intervals $\{\CIDS_t\}$ form a $(1-\alpha)$-confidence sequence of $\mu$:
\begin{equation}\label{eqn:cids}
    \CIDS_t = \left[ \frac{\sum_{i=1}^t \lambda_i X_i \pm \left(2/\alpha - 1 + \sigma^2 \sum_{i=1}^t \lambda_i^2 \right)  }{\sum_{i=1}^t \lambda_i  } \right].
\end{equation}
\end{theorem}

The straightforward proof of this theorem is in \cref{sec:pf}.

Now, we shall choose the coefficients $\{\lambda_t\}$ that appear in the theorem in order to optimize the interval widths $\{ |\CIDS_t| \}$. Our heuristic for optimizing the width is inspired by \citet[Equations (24--28)]{waudby2020estimating}; that is, we first fix a target time $t^\star$ and consider 
\begin{equation}\label{eqn:const-lambdas}
\lambda_1 = \lambda_2 = \dots =  \lambda^\star,    
\end{equation}
a constant sequence. After finding the $\lambda^\star$ that minimizes $|\CIDS_{t^\star}|$, we set $\lambda_{t^\star}$ to this value.
The detailed tuning procedure can be found in \cref{sec:ds-tuning}, where we show that
\begin{equation}\label{eqn:cids-heuristic-choice}
    \lambda_t = \sqrt{\frac{2/\alpha - 1}{\sigma^2 t}}
\end{equation}
is a prudent choice. Then, the width of the confidence sequence at time $t$ is
\begin{equation}\label{eqn:dsci-width}
    |\CIDS_{t}| =  \frac{ 2\sqrt{2/\alpha - 1} \, \sigma (1 +  \sum_i i^{-1})}{\sum_i i^{-1/2}} \asymp \frac{ \sqrt{2/\alpha - 1} \, \sigma \log t}{\sqrt{t}}.
\end{equation}

Let us briefly compare the $\mathcal{O}(\alpha^{-1/2})$ rate of width growth, and the $\Otilde(t^{-1/2})$ rate of width shrinkage we achieved in \eqref{eqn:dsci-width} with the well-known case of confidence intervals. Both the $\mathcal{O}(\alpha^{-1/2})$ rate of growth and the $\mathcal{O}(t^{-1/2})$ rate of shrinkage of the Chebyshev CIs \eqref{eqn:chebci}, which hold under a stronger assumption than our paper (i.e.\ independence and variance upper bound $\Var[X_i] \le \sigma^2$), are matched by our Dubins-Savage CS, up to the $\log t$ factor. It is worth remarking that the Chebyshev CIs $\{  \CIch_t \}$ \emph{never} form a confidence sequence at any level --- almost surely, there exists some $\CIch_{t_0}$ that does not contain $\mu$.

While the $\Otilde(t^{-1/2})$ rate of shrinkage cannot
be improved 
(which shall be discussed in \cref{sec:lb}), we shall see in the following sections that growth rates sharper than $\mathcal{O}(\alpha^{-1/2})$ can be achieved.
The sharper rates require eschewing the (weighted) empirical means, e.g.\ the $\frac{\sum \lambda_i X_i}{\sum \lambda_i}$ that centers the interval $\CIDS_t$ in \eqref{eqn:cids} above, because they have equally heavy tails as the observations $\{X_t\}$.


\section{Intermezzo: review of 
Ville's inequality} \label{sec:ville}

The remaining two types of confidence sequence in this paper are both based on the technique of constructing an appropriate pair of \emph{nonnegative supermartingales} \citep{howard2020time}. This powerful technique results in dramatically tighter confidence sequences compared to the previous approach \emph{\`a la} Dubins-Savage.

A stochastic process $\{ M_t \}_{t \in \mathbb N}$, adapted to the filtration $\{ \mathcal{F}_t \}_{t \in \mathbb N}$, is called a supermartingale if $\Exp[M_t \mid \mathcal{F}_{t-1}] \le M_{t-1}$ for all $t \in \mathbb N^+$. Since many of the supermartingales we are to construct are in an exponential, multiplicative form, we frequently use the following (obvious) lemma.

\begin{lemma}\label{lem:mult-nsm} Let $M_t = \prod_{i=1}^t m_i $ where each $m_i \ge 0$ is $\mathcal{F}_i$-measurable. Then, the process $\{ M_t \}$ is a nonnegative supermartingale if $\Exp[m_t\mid \mathcal{F}_{t-1}] \le 1$ for all $t \in \mathbb N^+$.
\end{lemma}

A remarkable property of nonnegative supermartingales is Ville's inequality \citep{ville1939etude}. It extends Markov's inequality from a single time to an infinite time horizon.

\begin{lemma}[Ville's inequality]\label{lem:vi}
Let $\{M_t\}_{t \in \mathbb N}$ be a nonnegative supermartingale with $M_0=1$. Then for all $\alpha \in (0, 1)$,
\begin{equation}\label{eqn:vil}
    \Pr[ \exists t \in \mathbb N^+, \; M_t \ge 1/\alpha ] \le \alpha.
\end{equation}
\end{lemma}
A very concise proof of \cref{lem:vi} can be found in, for example, \citet[Section 6.1]{howard2020time}.

\section{Confidence sequence
by self-normalization}

Our second confidence sequence comes from a predictable-mixing version of \citet[Proposition 12]{delyon2009exponential} and \citet[Lemma 3 (f)]{howard2020time}.
\begin{lemma}\label{lem:sn-martingale} Let $\{ \lambda_t \}_{t \in \mathbb N ^+}$ be any predictable process.
The following process is a nonnegative supermartingale:
\begin{equation}\label{eqn:sn-nsm}
    \Msn_t = \prod_{i=1}^t \exp \left(\lambda_i (X_i - \mu) - \frac{\lambda^2_i ( (X_i - \mu)^2 + 2 \sigma^2 )}{6} \right).
\end{equation}
\end{lemma}

The proof is in \cref{sec:pf}. We can obtain another nonnegative supermartingale by flipping $\{\lambda_t\}$ into $\{ -\lambda_t \}$ in \eqref{eqn:sn-nsm}.
Applying Ville's inequality (\cref{lem:vi}) on the two nonnegative supermartingales, we have the following result which is again proved in \cref{sec:pf}.

\begin{lemma}[Self-normalized \emph{anti}confidence sequence]\label{lem:anti-cisn} Let $\{ \lambda_t \}_{t \in \mathbb N ^+}$ be any predictable process. Define
\begin{gather}
    U_t^{+} = \frac{\sum_{i=1}^t \lambda_i^2 X_i}{3} - \sum_{i=1}^t \lambda_i, \quad U_t^{-} = \frac{\sum_{i=1}^t \lambda_i^2 X_i}{3} + \sum_{i=1}^t \lambda_i. \label{eqn:upm}
\end{gather}
We further define the interval $\aCISNp_{t}$ to be\footnote{When the term inside the square root is negative, by convention the interval is taken to be $\varnothing$.}
\begin{equation}\label{eqn:new-acisn+}
    \left[ \frac{ \left( U_t^{+} \right) \pm \sqrt{  \left( U_t^{+}  \right) ^2 -  \frac{2\sum \lambda_i^2}{3}  \left( \log(2/\alpha) - \sum\lambda_i X_i  + \frac{\sum\lambda_i^2 X_i^2 +  2\sigma^2 \sum\lambda_i^2 }{6} \right)  }    }{ \frac{\sum\lambda_i^2}{3} }  \right]
\end{equation}
(where each $\sum$ stands for $\sum_{i=1}^t$), and the interval $\aCISNn_{t}$ to be
\begin{equation}\label{eqn:new-acisn-}
    \left[ \frac{ \left( U_t^{-} \right) \pm \sqrt{  \left( U_t^{-}  \right) ^2 -  \frac{2\sum \lambda_i^2}{3}  \left( \log(2/\alpha) + \sum\lambda_i X_i  + \frac{\sum\lambda_i^2 X_i^2 +  2\sigma^2 \sum\lambda_i^2 }{6} \right)  }    }{ \frac{\sum\lambda_i^2}{3} }  \right].
\end{equation}
Then, both $\{ \aCISNp_{t} \}$ and $\{ \aCISNn_{t} \}$ form a $(1-\alpha/2)$-\emph{anti}confidence sequence for $\mu$. That is,
\begin{gather}
    \Pr\left[\forall t \in \mathbb N^+, \mu \notin {\aCISN}_{t}^+\right] \ge 1-\alpha/2,
    \\
    \Pr\left[\forall t \in \mathbb N^+, \mu \notin {\aCISN}_{t}^-\right] \ge 1-\alpha/2.
\end{gather}
\end{lemma}

Applying union bound on \cref{lem:anti-cisn} immediately gives rise to the following confidence sequence.

\begin{theorem}[Self-normalized confidence sequence] \label{thm:cisn}
The following random sets $\{\CISN_t\}$ form a $(1-\alpha)$-confidence sequence of $\mu$.
\begin{equation}
    \CISN_t = \mathbb R \setminus \left( \aCISNp_{t} \cup  \aCISNn_{t} \right),
\end{equation}
where $\aCISNp_{t}$ and $\aCISNn_{t}$ are defined in \eqref{eqn:new-acisn+} and \eqref{eqn:new-acisn-}.
\end{theorem}
It is not difficult to perform a cursory analysis on the topology of $\CISN_t$. Without loss of generality assume $\mu = 0$ since the method is translation invariant. When we take $\{\lambda_t\}$ to be a decreasing sequence, with high probability $\sum \lambda_i^2 X_i^2$ will be much smaller than $\sum \lambda_i$ in the long run, implying that $U_t^+ < 0$ while $U_t^- > 0$. Thus, with high probability,
\begin{equation}
    \max(\aCISNp_{t}) \le \frac{U_t^+ + |U_t^+|}{\sum \lambda_i^2/3} = 0 = \frac{U_t^- - |U_t^-|}{\sum \lambda_i^2/3}  \le \min(\aCISNn_{t}).
\end{equation}
\setlength{\columnsep}{16pt}%
\begin{wrapfigure}{r}{0.3\textwidth}
\centering
\includegraphics[width=0.3\textwidth]{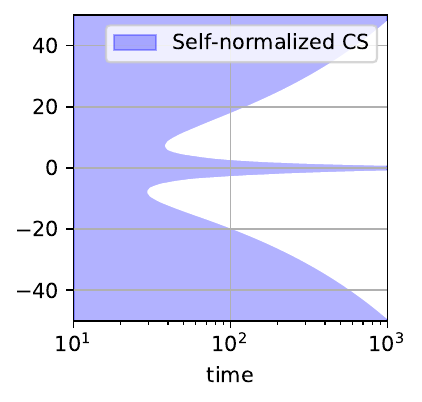}
\vspace{-30pt}
\caption{\small Self-normalized confidence sequence $\{ \CISN_t \}$ under $X_t \overset{\text{i.i.d.}}{\sim} \mathcal{N}({0,1})$, $\lambda_t = 1/\sqrt{t}$, and $\alpha = 0.05$. The trichotomous topology \eqref{eqn:cisn-decomp} manifests after around $t = 30$.} 
\label{fig:sncs}
\vspace{-60pt}
\end{wrapfigure}
Therefore, we expect the disjoint union of three intervals
\begin{multline}\label{eqn:cisn-decomp}
     \CISN_t = \left(-\infty, \min(\aCISNp_{t}) \right] \cup \\
    \left[\max(\aCISNp_{t}), \min(\aCISNn_{t}) \right] 
    \cup \left[ \max(\aCISNn_{t}), \infty  \right) \\
    =: \LSN_t \cup \MSN_t \cup \USN_t
\end{multline}
to be the typical topology of $\CISN_t$ for large $t$. Indeed, we demonstrate this with the a simple experiment under $X_t \overset{\text{i.i.d.}}{\sim} \mathcal{N}(0,1)$ and $\lambda_t = 1/\sqrt{t}$, plotted in \cref{fig:sncs}.

We now come to the question of choosing the predictable process $\{\lambda_t\}$ to optimize the confidence sequence. Since the set $\CISN_t$ always has infinite Lebesgue measure, a reasonable objective is to ensure $\min(\aCISNp_{t}) \to -\infty$, $\max(\aCISNn_{t}) \to \infty$, and make the middle interval $\MSN_t$ as narrow as possible. We resort to the same heuristic approach as in the Dubins-Savage case when optimizing $|\MSN_t|$, which is detailed in the \cref{sec:sn-tuning}. The result of our tuning is
\begin{equation}\label{eqn:sncs-heuristic}
    \lambda_t = \sqrt{\frac{6t \log(2/\alpha)}{t( \sum_{i=1}^{t-1} X_i^2  +2 \sigma^2 t) - (\sum_{i=1}^{t-1} X_i)^2}}.
\end{equation}
Indeed, $\min(\aCISNp_{t}) \to -\infty$, $\max(\aCISNn_{t}) \to \infty$ almost surely if we set $\{\lambda_t\}$ as above. We remark that the removal of the ``spurious intervals'' $\LSN_t$ and $\USN_t$ is easily achieved. For example, split $\alpha = \alpha' + \alpha''$ such that $\alpha'' \ll \alpha'$. First, construct a $(1-\alpha'')$-CS $\{ \CI_t^\circ \}$ that does not have such spurious intervals -- e.g., the Dubins-Savage CS \eqref{eqn:cids}. Next, construct the self-normalized CS $\{ \CISN_t \}$ at $1-\alpha'$ confidence level. A union bound argument yields that $\{ \CI_t^\circ \cap \CISN_t  \}$ is a $(1-\alpha)$-CS, and the intersection with $\{ \CI_t^\circ \}$ helps to get rid of the spurious intervals $\LSN_t$ and $\USN_t$.

\section{Confidence sequence via Catoni supermartingales}\label{sec:catoni}

Our last confidence sequence is inspired by \citet{catoni2012challenging}, where under only the finite variance assumption, the author constructs an M-estimator for mean that is $\mathcal{O}(\log(1/\alpha))$-close to the true mean with probability at least $1-\alpha$; hence a corresponding $(1-\alpha)$-CI whose width has $\mathcal{O}(\log(1/\alpha))$ growth rate exists; cf.\ \eqref{eqn:chebci}. We shall sequentialize the idea of  \citet{catoni2012challenging} via constructing two nonnegative supermartingales which we shall call the \emph{Catoni Supermartingales}.

Following \citet[Equation (2.1)]{catoni2012challenging}, we say that $\phi:\mathbb R \to \mathbb R$ is a \emph{Catoni-type influence function}, if it is increasing and $-\log(1 - x + x^2/2) \le \phi(x) \le \log(1 + x + x^2/2)$. A simple Catoni-type influence function is 
\begin{equation}
    \phi (x) = \begin{cases} \log(1 + x + x^2/2), & x \ge 0; \\ -\log(1 - x + x^2/2), & x < 0. \end{cases}
\end{equation}
\begin{lemma}[Catoni supermartingales]\label{lem:cat-sm} Let $\{ \lambda_t \}_{t \in \mathbb N ^+}$ be any predictable process, and
let $\phi$ be a Catoni-type influence function. The following processes are nonnegative supermartingales:
\begin{gather}\label{eqn:cat-sm}
    \Mcat_t = \prod_{i=1}^t \exp \left(\phi(\lambda_i (X_i  -\mu)) - \lambda^2_i \sigma^2/2 \right);\\  \Ncat_t = \prod_{i=1}^t \exp \left(-\phi(\lambda_i (X_i  -\mu)) - \lambda^2_i \sigma^2/2 \right) .
\end{gather}
\end{lemma}
This lemma is proved in \cref{sec:pf}. We now remark on the ``tightness" of \cref{lem:cat-sm}. On the one hand, it is tight in the sense that the pair of processes make the fullest use of \cref{ass:var} to be supermartingales, which we formalize in \cref{sec:pf} with \cref{prop:tightness-cat-nsm}; on the other hand, a slightly tighter (i.e.\ larger) pair of supermartingales do exist, but are not as useful in deriving CS (see \cref{sec:tighter}).
In conjunction with Ville's inequality (\cref{lem:vi}), \cref{lem:cat-sm} immediately gives a confidence sequence.

\begin{theorem}[Catoni-style confidence sequence]\label{thm:cs-cat}
Let $\{ \lambda_t \}_{t \in \mathbb N ^+}$ be any predictable process, and
let $\phi$ be a Catoni-type influence function. The following intervals $\{\CICT_t\}$ form a $(1-\alpha)$-confidence sequence for $\mu$:
\begin{equation}\label{eqn:catoni-cs}
    \CICT_t =  \bigg\{  m \in \mathbb R :  -  \frac{\sigma^2 \sum_{i=1}^t \lambda_i^2}{2} - {\log(2/\alpha)}  \le \sum_{i=1}^t  \phi(\lambda_i (X_i  - m ))  \le \frac{\sigma^2 \sum_{i=1}^t \lambda_i^2 }{2} + {\log(2/\alpha)} \bigg\}.
\end{equation}
\end{theorem}

Although this confidence sequence lacks a closed-form expression, it is easily computed using root-finding methods since the function $m \mapsto \sum_{i=1}^t  \phi(\lambda_i (X_i  - m ))$ is monotonic. A preliminary experiment is shown in \cref{fig:cat}.

We follow \citet[Proposition 2.4]{catoni2012challenging} in choosing the appropriate $\{\lambda_t\}$ sequence to optimize the interval widths,
\begin{equation}\label{eqn:catoni-opt-lambda}
    \lambda_t = \sqrt{ \frac{2\log(2/\alpha)}{t(\sigma^2 + \eta^2_t)} } \; \text{where } \eta_t = \sqrt{\frac{2\sigma^2 \log(2/\alpha)}{t - 2\log(2/\alpha)}}.
\end{equation}

\begin{wrapfigure}{r}{0.3\textwidth}
\centering
\includegraphics[width=0.3\textwidth]{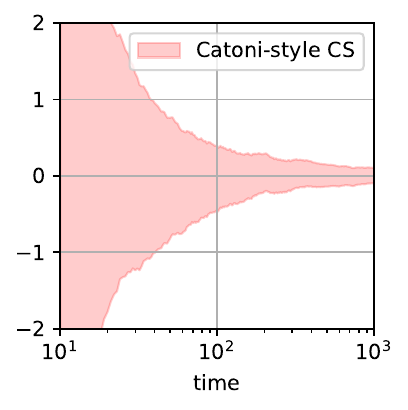}
\vspace{-30pt}
\caption{\small Catoni-style confidence sequence $\{ \CICT_t \}$ under $X_t \overset{\text{i.i.d.}}{\sim} \mathcal{N}({0,1})$, $\lambda_t = \min\{1/\sqrt{t} , 0.1 \}$, and $\alpha = 0.05$. Notice the difference in scale of the $y$-axis from Figure~\ref{fig:sncs}.} 
\label{fig:cat}
\vspace{-10pt}
\end{wrapfigure}

We shall show very soon in extensive experiments (\cref{sec:exp}) that the Catoni-style confidence sequence performs remarkably well controlling the width $|\CICT_t|$, not only outperforming the previously introduced two confidence sequences, but also matching the best-performing confidence sequences and even confidence intervals in the literature, many of which require a much more stringent distributional assumption. On the theoretical side, we establish the following nonasymptotic concentration result on the width $|\CICT_t|$.
\begin{theorem}\label{thm:catoni-tight} 
Suppose the coefficients $\{ \lambda_t \}_{t \in \mathbb N^+}$ are nonrandom and let $0 < \varepsilon < 1$. Suppose further that
\begin{equation}\label{eqn:quadratic-condition}
    \left(\sum_{i=1}^t \lambda_i \right)^2 - 2\left(\sum_{i=1}^t \lambda_i^2\right) \left(\sigma^2 \sum_{i=1}^t \lambda_i^2 + \log(2/\varepsilon ) + \log(2/\alpha) \right) \ge 0.
\end{equation}
Then, with probability at least $1-\varepsilon$,
\begin{equation}
    |\CICT_t|\le \frac{ 4(\sigma^2 \sum_{i=1}^t \lambda_i^2  + \log(2/\varepsilon) + \log(2/\alpha)) }{ \sum_{i=1}^t \lambda_i }.
\end{equation}
\end{theorem}

The proof of the theorem is inspired by the deviation analysis of the nonsequential Catoni estimator \citet[Proposition 2.4]{catoni2012challenging} itself, and can be found in \cref{sec:pf}.

We remark that \eqref{eqn:quadratic-condition} is an entirely deterministic inequality when $\{\lambda_i \}_{i \in \mathbb N^+}$ are all nonrandom. When $\lambda_t = \Otheta (t^{-1/2})$, which is the case for \eqref{eqn:catoni-opt-lambda}, the condition \eqref{eqn:quadratic-condition} holds for large $t$ since $\sum_{i=1}^t \lambda_i = \Otheta(\sqrt{t})$ while $\sum_{i=1}^t \lambda_i^2$ grows logarithmically. This gives us the following qualitative version of \cref{thm:catoni-tight}.

\begin{corollary}\label{cor:catoni-tight} Suppose $\lambda_t  = \Otheta (t^{-1/2})$ and is nonrandom, such as in~\eqref{eqn:catoni-opt-lambda}. Then if $t > \polylog(1/\varepsilon, 1/\alpha)$, with probability at least $1 - \varepsilon$,
\begin{equation}
     |\CICT_t|\le \Otilde( t^{-1/2}(\log(1/\varepsilon) + \log(1/\alpha))).
\end{equation}
Here, the notation $\Otilde$ only hides logarithmic factors in $t$.
\end{corollary}
We hence see that the Catoni-style confidence sequence enjoys the $\Otilde(t^{-1/2})$ and $\mathcal{O}(\log(1/\alpha))$ near-optimal rates of shrinkage and growth. If we do not ignore the logarithmic factors in $t$, for example, taking $\lambda_t \asymp 1 / \sqrt{t\log t}$, the rate shall be $|\CICT_t| = \mathcal{O} (\log \log t \sqrt{\log t / t})$ (cf.\ \citet[Table 1]{waudby2020estimating}).

It is now natural to ask whether the Catoni-style CS can obtain the law-of-the-iterated-logarithm rate $\Theta(\sqrt{\log \log t/t})$. This cannot be achieved by tuning the sequence $\{ \lambda_t \}$ alone \citep[Table 1]{waudby2020estimating}, but can be achieved using a technique called \emph{stitching} \citep{howard2021time}.

\begin{corollary}\label{cor:loglogt} Let $\{ {\CICT_t}(\Lambda, \alpha)  \}$ denote the Catoni-style confidence sequence as in \eqref{eqn:catoni-cs} with $\sigma^2 = 1$, \emph{constant} $\{\lambda_t\}$ sequence $\lambda_1 = \lambda_2 = \dots = \Lambda$, and error level $\alpha$. Then, let $t_j = \e^j$, $\alpha_j = \frac{\alpha}{(j+2)^2}$, and $\Lambda_j = \sqrt{ \log(2/\alpha_j)  \e^{- j} }$. The following \emph{stitched Catoni-style confidence sequence}
\begin{equation}
   \CIstitch_t = {\CICT_t}(\Lambda_j, \alpha_j), \quad \text{for } t_j \le t < t_{j+1} 
\end{equation}
is a $(1-\alpha)$-CS for $\mu$ and satisfies
\begin{equation}\label{eqn:stitchresult}
\Pr \left[    |\CIstitch_t| \le 8(\e + 1) \sqrt{ \frac{\log(2/\alpha) + 2\log\log\e^2 t}{t} }  \right] \ge 1 -\alpha/4
\end{equation}
when $t > \polylog(1/\alpha)$.
\end{corollary}
$\{ \CIstitch_t  \}$ forms a $(1-\alpha)$-CS because of a union bound over $\sum_{j \ge 0} \alpha_j \le \alpha$. The width in \eqref{eqn:stitchresult} matches both the $\Theta(\sqrt{\log(1/\alpha)})$ lower bound on the growth rate, and the $\Theta(\sqrt{\log \log t/t})$ lower bound on the shrinkage rate (which we shall present soon in \cref{sec:lb}).
It pays the price of a larger multiplicative constant to achieve the optimal shrinkage rate, so we only recommend it when long-term tightness is of particular interest.
The proof of this corollary is in \cref{sec:pf}. 

\begin{remark}\normalfont
While the union bound argument of \cref{cor:loglogt} asymptotically \emph{improves} an existing $t^{-1/2} \polylog(t)$ CS to a $\sqrt{\log \log t/t}$ CS, there is a related (but much less involved) idea to \emph{construct} a CS from a sequence of CIs: split $\alpha = \sum_{t=1}^\infty \alpha_t$ and define the CS as the CI at time $t$ with error level $\alpha_t$. This, however, leads to poor performance. Because the events $\{\mu \in \mathrm{CI}_t\}$ and $\{\mu \in \mathrm{CI}_{t+1}\}$ are highly dependent, making the union bound over $t \in \mathbb N^+$ very loose. 
In \cref{fig:cs}, we visually demonstrate the Catoni CIs with $\alpha_t = \alpha/[t(t + 1)]$ (forming what we call the ``trivial Catoni CS''), versus our supermartingale-based Catoni-style CS \eqref{eqn:catoni-cs}.
\end{remark}
\begin{figure}[h]
  \begin{minipage}[c]{0.7\textwidth}
    \includegraphics[width=\textwidth]{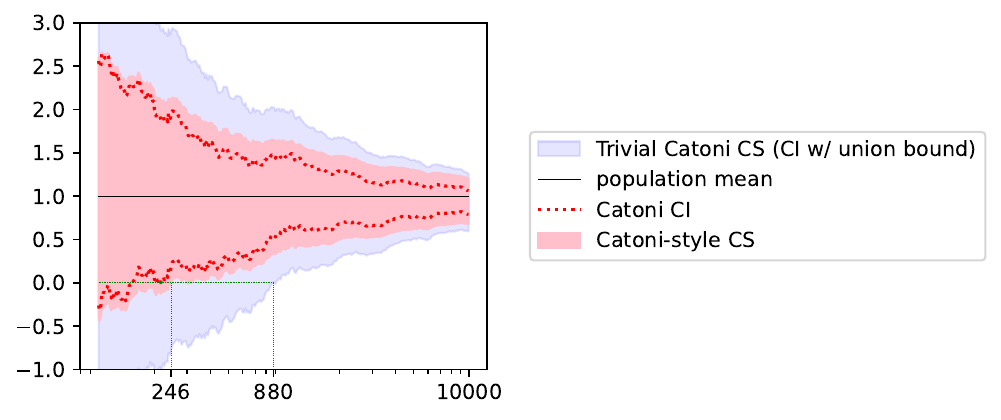}
  \end{minipage}\hfill
  \begin{minipage}[c]{0.27\textwidth}
    \caption{\small 
    To achieve the same level of lower tightness (e.g.\ when lower confidence bound surpasses 0), the trivial Catoni CS needs a sample of size 880, about 4 times the Catoni-style CS which only takes 246.
    } \label{fig:cs}
  \end{minipage}
\end{figure}

\section{Lower bounds}\label{sec:lb}

For the sake of completeness, we now discuss the lower bounds of confidence sequences for mean estimation. We first introduce the following notion of \emph{tail symmetry}, which is a standard practice when constructing two-sided confidence intervals and sequences.

\begin{definition}[tail-symmetric CI/CS]Let $\mathcal{Q}$ be a family of distributions over $\mathbb R^{\mathbb N^+}$ such that for each $Q \in \mathcal{Q}$, $\Exp_Q[X_t] = \mu_Q$ is independent of $t$.
A sequence of data-dependent intervals $\{ [\lw_t, \up_t] \}_{t \in \mathbb N^+}$ where $\lw_t, \up_t: \mathbb R^t \to \mathbb R$ is called a sequence of tail-symmetric $(1-\alpha)$-CI over $\mathcal{Q}$ if for any $t \in \mathbb N^+$,
\begin{equation}
   \forall Q \in \mathcal{Q},  \Prw_{\{X_i\} \sim Q} \left[ \mu_Q \le \lw_t(X_1, \dots, X_t)  \right] \le \frac{\alpha}{2} ,  \Prw_{\{X_i\} \sim Q} \left[ \mu_Q \ge \up_t(X_1, \dots, X_t)  \right] \le \frac{\alpha}{2}.
\end{equation}
Further, it is called a tail-symmetric $(1-\alpha)$-CS over $\mathcal{P}$ if
\begin{multline}
   \forall Q \in \mathcal{Q}, \quad  \Prw_{\{X_i\} \sim Q}  \left[ \forall t \in \mathbb N^+, \; \mu_Q \le \lw_t(X_1, \dots, X_t)  \right] \le \frac{\alpha}{2}, \\  \Prw_{\{X_i\} \sim Q} \left[ \forall t \in \mathbb N^+ ,  \; \mu_Q \ge \up_t(X_1, \dots, X_t)  \right] \le \frac{\alpha}{2}.
\end{multline}
\end{definition}

The following lower bound of minimax nature, akin to \citet[Section 6.1]{catoni2012challenging}), characterizes the minimal growth rate of confidence intervals (hence also confidence sequences) when $\alpha$ is small. Its proof shall be found in \cref{sec:pf}.

\begin{proposition}[Gaussian lower bound]\label{prop:gaussianlowerbound} 
We define
\begin{equation}
    \mathcal{M}_t(\alpha, \varepsilon) = \inf_{[  \lw_t, \up_t ]} \sup_{Q\in \mathcal{Q}_{\sigma^2}} w^{Q, \varepsilon}_t,
\end{equation}
where the supremum is taken over $\mathcal{Q}_{\sigma^2}$, the set of all distributions of $\{X_t\}$ satisfying \cref{ass:mean} (where $\mu$ ranges over $\mathbb R$) and \cref{ass:var} (where $\sigma^2$ is fixed), the infimum over all tail-symmetric $(1-\alpha)$-confidence intervals $\{[  \lw_t, \up_t ]\}$ over $\mathcal{Q}_{\sigma^2}$, and $w^{Q, \varepsilon}_t$ is the number such that $\Prw_{\{X_i\} \sim Q}[ \up_t - \lw_t \le w^{Q, \varepsilon}_t ] = 1-\varepsilon$.

Let $\Phi^{-1}$ be the quantile function of the standard normal distribution. Then, as long as $\alpha + 2\varepsilon < 1$,
\begin{multline}
     \mathcal{M}_t(\alpha, \varepsilon) \ge \frac{2\sigma}{\sqrt{t}}\Phi^{-1}(1-\alpha/2 - \varepsilon) 
     \\
     = \frac{2\sigma}{\sqrt{t}} \left(\sqrt{\log\frac{4}{(\alpha+2\varepsilon)^2} - \log \log\frac{4}{(\alpha+2\varepsilon)^2} - \log(2\pi) } + \mathcal{O}(1) \right).
\end{multline}
Here $\mathcal{O}(1)$ is with respect to $\alpha, \varepsilon \to 0$.
\end{proposition}

The next lower bound, due to the law of the iterated logarithm (LIL), lower bounds the shrinkage rate of confidence sequences as $t \to \infty$. The proof is again delayed to \cref{sec:pf}.


\begin{proposition}[LIL lower bound]\label{prop:lil} Let $P$ be a distribution on $\mathbb R$ with mean $\mu$ and variance $\sigma^2$. Let $\{ \CI_t \}$ be a $(1-\alpha)$-confidence sequence for $X_t \overset{\mathrm{iid}}{\sim} P$ such that $\hmu_t \in \CI_t$. Then,
\begin{equation}
\Pr \left[\limsup_{t \to \infty} \frac{|\CI_t|}{\sigma}\sqrt{\frac{t}{2\log\log t}} \ge 1 \right] \ge 1- \alpha.
\end{equation}
\end{proposition}

\begin{remark}\normalfont
The assumption $\hmu_t \in \CI_t$ is true for many existing confidence sequences for mean estimation in the literature \cite{darling1967confidence,jennison1989interim,howard2021time}, meaning that our CS that matches this lower bound \eqref{eqn:stitchresult} is fundamentally not worse than them even under the much weaker \cref{ass:var}. While assuming $\hmu_t \in \CI_t$ does not encompass all the CSs in the literature, it can be relaxed by assuming instead that \emph{there exists an estimator $\htheta_t \in \CI_t$ that follows the law of the iterated logarithm}. For example, it is known that the weighted empirical average satisfies the LIL \cite{teicher1974law}, which implies that all of the predictable mixture confidence sequences due to \citet{waudby2020estimating}, as well as our \cref{thm:cids}, are subject to a similar LIL lower bound. Relatedly, the LIL is also satisfied by some M-estimators \cite{he1995law,brunel2019nonasymptotic}. However, these LIL-type results are only valid under constant weight multipliers (in our parlance, that is the sequence $\{\lambda_t \}$ is constant, in which case our CSs do not shrink), and hence the M-estimators for $\mu$ can be inconsistent, the limit of the LIL-type convergence being some value other than $\mu$. The search for new LIL-type results for \emph{consistent} M-estimators under decreasing $\{ \lambda_t \}$ sequence, e.g.\ the zero of $\sum_{i=1}^t  \phi(\lambda_i (X_i  - m ))$ which is included in the Catoni-style CS \eqref{eqn:catoni-cs}, shall stimulate our future study.
\end{remark}

\section{Experiments}\label{sec:exp}
We first examine the empirical cumulative miscoverage rates of our confidence sequences as well as Catoni's confidence interval. These are the frequencies at time $t$ that any of the intervals $\{ \CI_i \}_{1 \le i \le t}$ does not cover the population mean $\mu$, under 2000 (for all CSs) or 250 (for the Catoni CI\footnote{due to its inherent non-sequentializable computation}) independent runs of i.i.d.\ samples of size $t =800$ from a Student's t-distribution with 3 degrees of freedom, randomly centered and rescaled to variance $\sigma^2 = 25$. Its result, in \cref{fig:miscoverage}, shows the clear advantage of CSs under continuous monitoring as they never accumulate error more than the preset $\alpha$, unlike the Catoni CI whose cumulative miscoverage rate goes beyond $\alpha$ early on. In fact, a similar experiment in the light-tailed regime by \citet[Figure 1(b)]{howard2021time} shows that the cumulative miscoverage rate of CIs will grow to 1 if we extend the sequential monitoring process indefinitely.

\begin{figure}[h]
  \begin{minipage}[c]{0.63\textwidth}
    \includegraphics[width=\textwidth]{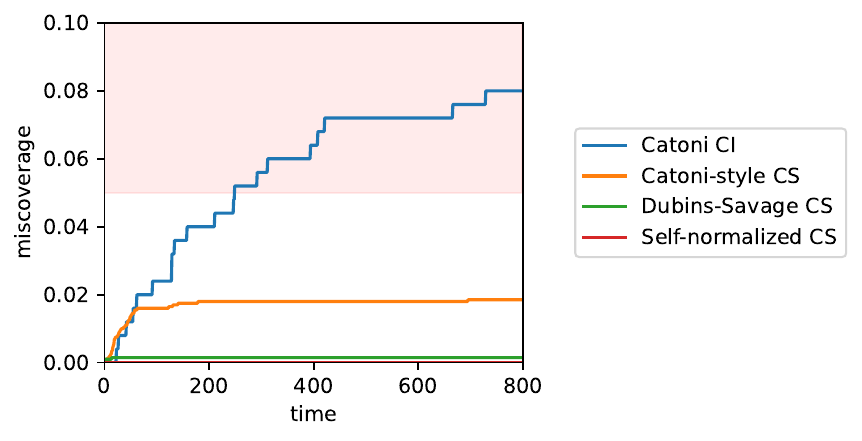}
  \end{minipage}\hfill
  \begin{minipage}[c]{0.33\textwidth}
    \caption{\small 
    Cumulative miscoverage rates when continuously monitoring CSs and CI under $\mathrm{t}_3$ distribution, which (provably) grow without bound for the Catoni CI, but are guaranteed to stay within $\alpha = 0.05$ for CSs.
    } \label{fig:miscoverage}
  \end{minipage}
\end{figure}

\begin{figure}[!h]
    \centering
    \subfloat[Growths under $\mathrm{t_3}$ distribution.]{{\includegraphics[height=0.32\columnwidth]{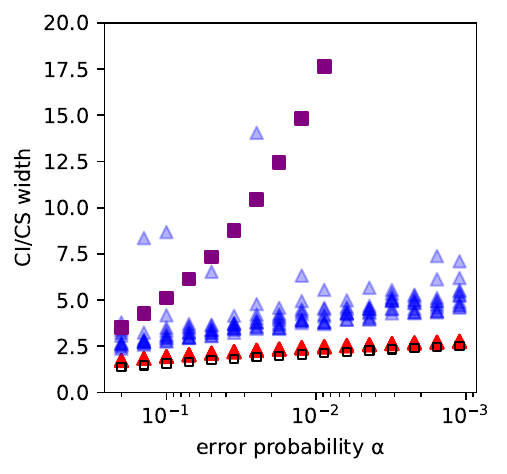} } \label{fig:width-ht}}%
    \subfloat[Growths under normal distribution.]{{\includegraphics[height=0.32\columnwidth]{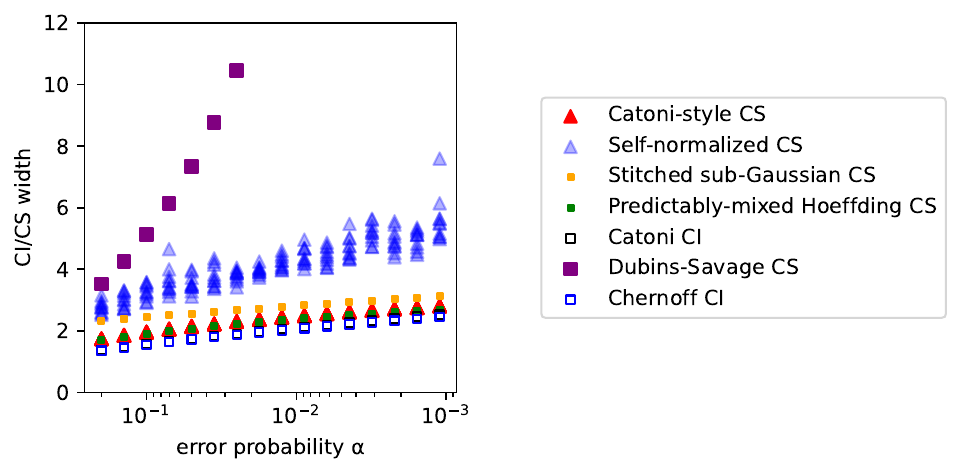} }  \label{fig:width-gaus} }%
    \caption{\small Comparison of CI/CS growth rates at $t =250$. In both figures, triangular markers denote random widths (for which we repeat 10 times to let randomness manifest), and square markers deterministic widths; hollow markers denote the widths of CIs, while filled markers the widths of CSs. Our Catoni-style CS is among the best CSs (even CIs) in terms of tightness under small error probability $\alpha$, in heavy and light tail regimes alike. In the right figure, note the overlap of the Chernoff-CI with Catoni-CI, as well as that of the Hoeffding-type subGaussian CS with the Catoni-style CS.}%
\end{figure}

We then compare the confidence sequences in terms of their growth rates. That is, we shall take decaying values of error level $\alpha$ and plot the length of the CSs, with the corresponding $\{\lambda_t\}$ sequences \eqref{eqn:cids-heuristic-choice}, \eqref{eqn:sncs-heuristic}, and \eqref{eqn:catoni-opt-lambda}\footnote{Actually we use $\{ \lambda_{\max\{t, 9\}} \}$ for the Catoni-style CS to facilitate the root-finding in \eqref{eqn:catoni-cs}.} we choose. We draw a $t = 250$ i.i.d.\ sample from the same Student's t-distribution as above ($\sigma^2 =25$). The Dubins-Savage CS has deterministic interval widths, while the self-normalized (we only consider $|\MSN_t|$) and Catoni-style CS both have random interval widths, for which we repeat the experiments 10 times each. We add the Catoni CI for the sake of reference. The comparison of widths is exhibited in \cref{fig:width-ht}.

We observe from the graph that the self-normalized CS and the Catoni-style CS both exhibit restrained growth of interval width when $\alpha$ becomes small. On the other hand, the Dubins-Savage CS, with its super-logarithmic $\mathcal{O}(\alpha^{-1/2})$ growth, perform markedly worse in contrast to those with logarithmic growth.

We run the same experiment again, this time with Gaussian data with variance $\sigma^2 = 25$ and we add two CSs and one CI for subGaussian random variables with variance factor $\sigma^2$ from previous literature for comparison. First, the stitched subGaussian CS \citep[Equation (1.2)]{howard2021time} which we review in \cref{prop:stitch};
Second, the predictably-mixed Hoeffding CS
\begin{equation}
    \mathrm{CI}^{\textsf{PM-H}}_t = \left[ \frac{\sum_{i=1}^t \lambda_i X_i \pm (\sigma^2 \sum_{i=1}^t \lambda_i^2 / 2 + \log(2/\alpha))}{\sum_{i=1}^t \lambda_i} \right]
\end{equation}
with \eqref{eqn:catoni-opt-lambda}; Third, the standard subGaussian Chernoff CI from \eqref{eqn:chernoff}. Recall that all three of the above bounds are not valid under only a finite variance assumption, but require a subGaussian moment generating function. This extended comparison is plotted in \cref{fig:width-gaus}.

Next, we examine the rates of shrinkage as $t \to \infty$. We sequentially sample from (i) a Student's t-distribution with 5 degrees of freedom, randomly centered and rescaled to variance $\sigma^2 = 25$, (ii) a normal distribution with variance $\sigma^2 = 25$, at each time calculating the intervals. We include again the Catoni's CI in both cases, the three subGaussian CI/CSs in the Gaussian case. The evolution of these intervals is shown in \cref{fig:shrink,fig:shrink-2}, and their widths in \cref{fig:shrink-len,fig:shrink-len-2}.

\begin{figure}[!h]
    \centering
    \subfloat[Comparison of intervals under \newline $\mathrm{t}_3$ distribution.]{{\includegraphics[height=0.375\columnwidth]{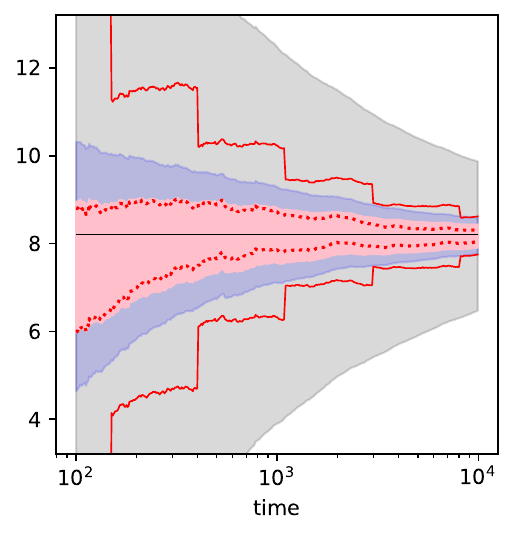} } \label{fig:shrink}}%
    \subfloat[Comparison of intervals under \newline normal distribution.]{{\includegraphics[height=0.375\columnwidth]{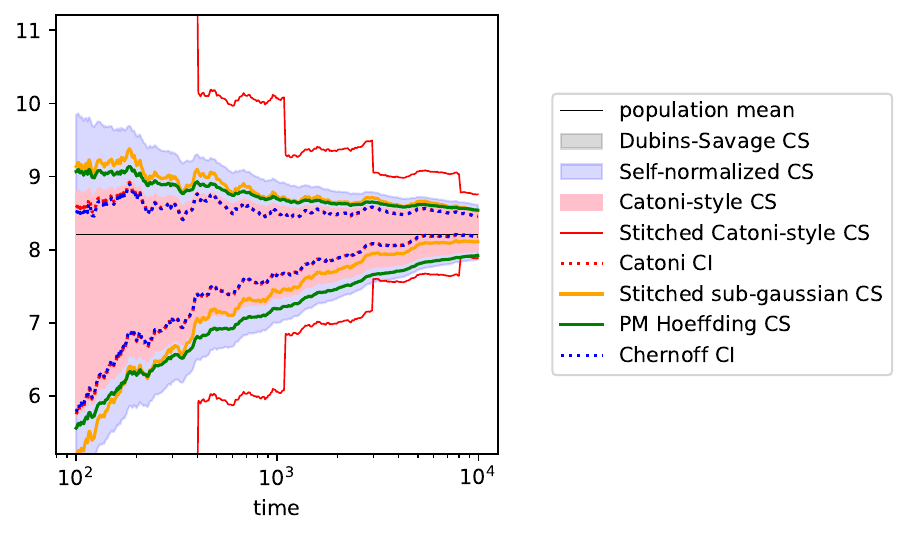} }  \label{fig:shrink-2} }%
    \newline
    \subfloat[Width shrinkage curves under \newline $\mathrm{t}_3$ distribution.]{{\includegraphics[height=0.375\columnwidth]{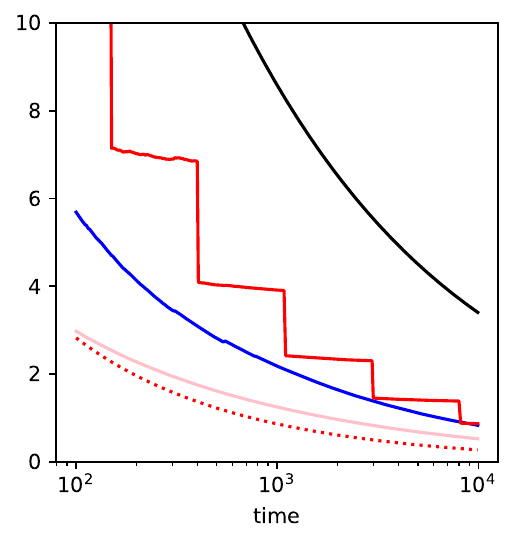} } \label{fig:shrink-len}}%
    \subfloat[Width shrinkage curves under \newline normal distribution.]{{\includegraphics[height=0.375\columnwidth]{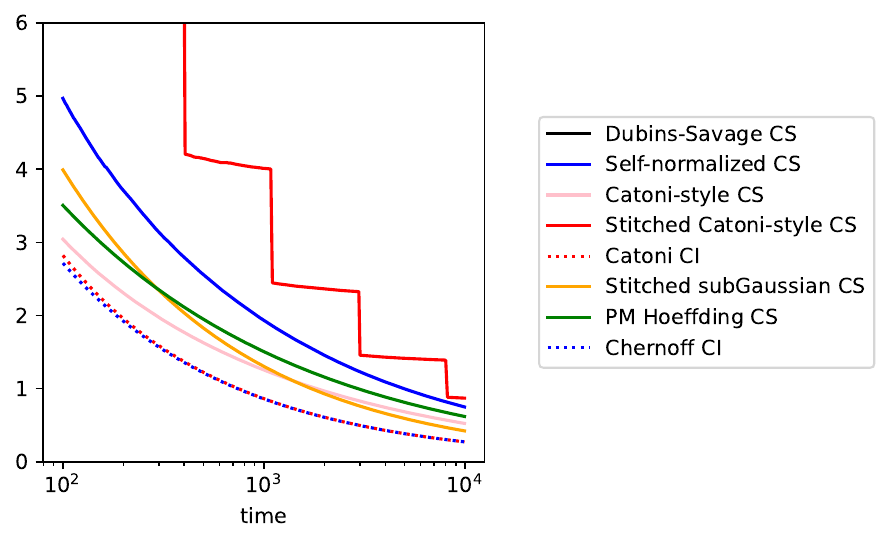} }  \label{fig:shrink-len-2} }%
    \caption{\small Comparison of CI/CS shrinkage rates at $\alpha = 0.05$. In the heavy-tailed regime, our Catoni-style CS performs markedly better than the other two CSs, and is close to the Catoni CI. In the Gaussian setting, our Catoni-style CS is approximately at a same caliber as the best subGaussian CSs in the literature. There appears no benefit in using the minimax optimal stitched Catoni-style CS for sample sizes within 10,000 due to its large constant. These new observations appear to be practically useful.}%
\end{figure}

It is important to remark that in the particular instantiation of the random variables that were drawn for the runs plotted in these figures, the Catoni CI seems to always cover the true mean; however, we know for a fact (theoretically from the law of the iterated logarithm for M-estimators \citep{brunel2019nonasymptotic}; empirically from \cref{fig:miscoverage}) that the Catoni CI will eventually miscover with probability one, and it will in fact miscover infinitely often, in every single run. When the first miscoverage exactly happens is a matter of chance (it could happen early or late), but it will almost surely happen infinitely~\citep{howard2021time}. Thus, the Catoni CI cannot be trusted at data-dependent stopping times, as encountered via continuous monitoring in sequential experimentation, but the CSs can. The price for the extra protection offered by the CSs is in lower order terms (polylogarithmic), and the figures suggest that it is quite minimal, the Catoni-CS being only ever so slightly wider than the Catoni CI.

\section{Heteroscedastic and infinite variance data}\label{sec:inf-var-in-text}
In lieu of \cref{ass:var}, we can consider a much more general setting that encompasses data drawn from distributions \emph{without} a finite variance, e.g.\ Pareto distribution or stable distribution with index in $(1,2)$, and possibly those that are \emph{increasing} in scale.
\begin{assumption}\label{ass:pthmoment-bound-het}
The process is conditionally $L^p$ with an upper bound, known \emph{a priori}, on the conditional central $p$\textsuperscript{th} moment:
    \begin{equation}\label{eqn:pthmoment-bound-het}
        \forall t\in\mathbb N^+, \quad \Exp[|X_t -\mu|^p  \mid \mathcal{F}_{t-1}] \le v_t, \quad (1<p\le 2)
\end{equation}
where $\{v_t\}_{t \in \mathbb N^+}$ is a predictable, nonnegative process.
\end{assumption}

When $p=2$, all of the three confidence sequences can extend naturally to handle such scenario of heteroscedasticity. We leave the details of the heteroscedastic versions of the Dubins-Savage CS and the self-normalized CS to \cref{sec:het-sce}. For the infinite variance case $1 < p < 2$, the generalization of the Dubins-Savage inequality in this infinite variance regime by \citet{khan2009p} can easily be used to construct a confidence sequence under \cref{ass:pthmoment-bound-het}, extending our \cref{thm:cids}. However, due to the relatively unsatisfactory performance of the Dubins-Savage CS, we do not elaborate upon this extension.

Let us focus primarily on extending our Catoni-style CS in \cref{thm:cs-cat} to \cref{ass:pthmoment-bound-het} in this section. To achieve this, we resort to an argument similar to the generalization of the Catoni CI by \citet{chen2021generalized}. 

We say that $\phi_p:\mathbb R \to \mathbb R$ is a $p$-Catoni-type influence function, if it is increasing and $-\log(1 - x + |x|^p/p) \le \phi_p(x) \le \log(1 + x + |x|^p/p)$. A simple example is 
\begin{equation}
    \phi_p (x) = \begin{cases} \log(1 + x + x^p/p), & x \ge 0; \\ -\log(1 - x + (-x)^p/p), & x < 0. \end{cases}
\end{equation}
\begin{lemma}[$p$-Catoni supermartingales]\label{lem:alphacat-sm}
Let $\phi_p$ be a $p$-Catoni-type influence function. Under \cref{ass:mean} and \cref{ass:pthmoment-bound-het}, the following processes are nonnegative supermartingales,
\begin{gather}\label{eqn:alphacat-sm}
    {\Mcat_t}^{p} = \prod_{i=1}^t \exp \left(\phi_p(\lambda_i (X_i  -\mu)) - \lambda^p_i v_i/p \right);
    \\  
    {\Ncat_t}^{p} = \prod_{i=1}^t \exp \left(-\phi_p(\lambda_i (X_i  -\mu)) - \lambda^p_i v_i/p \right) .
\end{gather}
\end{lemma}
The proof is straightforwardly analogous to the one of \cref{lem:cat-sm}. The corresponding CS can be easily expressed akin to \cref{thm:cs-cat}.
\begin{theorem}[$p$-Catoni-style confidence sequence]\label{thm:alpha-cs-cat}
Let $\phi_p$ be a $p$-Catoni-type influence function. Under \cref{ass:mean} and \cref{ass:pthmoment-bound-het}, the following intervals $\{{\CICT_t}^p\}$ form a $(1-\alpha)$-confidence sequence of $\mu$:
\begin{multline}\label{eqn:alphacatoni-cs}
    {\CICT_t}^p =  \Bigg\{  m \in \mathbb R :  -  \frac{ \sum_{i=1}^t v_i\lambda_i^p}{p} - {\log(2/\alpha)}   \le \sum_{i=1}^t  \phi_p(\lambda_i (X_i  - m ))  \le \frac{ \sum_{i=1}^t v_i\lambda_i^p }{p} + {\log(2/\alpha)} \Bigg\}.
\end{multline}
\end{theorem}
\citet{chen2021generalized} point out that in the i.i.d.\ case (i.e.\ assuming $v_t = v$ for all $t$ in \cref{ass:pthmoment-bound-het}), the asymptotically optimal choice for the rate of decrease of $\{\lambda_t\}$, when working with this $L^p$ set-up, would be $\lambda_t \asymp t^{-1/p}$. Specifically, in \citet[Proof of Theorem 2.6]{chen2021generalized}, the authors recommend the tuning
\begin{equation}\label{eqn:chen-lambda}
    \lambda_t = \frac{1}{2}\left( \frac{2p \log(2/\alpha)}{tv} \right)^{1/p}
\end{equation}
to optimize their CI. We adopt exactly the same tuning \eqref{eqn:chen-lambda} in our experiment, shown in \cref{fig:catp}, with i.i.d., infinite variance Pareto data. Indeed, employing $\lambda_t \asymp t^{-1/p}$ also leads to a width concentration bound of optimal shrinkage rate $t^{-(p-1)/p}$, similar to \cref{thm:catoni-tight} and proved in \cref{sec:pf}.

\begin{figure}[h]
\vspace{-0.1cm}
  \begin{minipage}[c]{0.35\textwidth}
    \includegraphics[width=\textwidth]{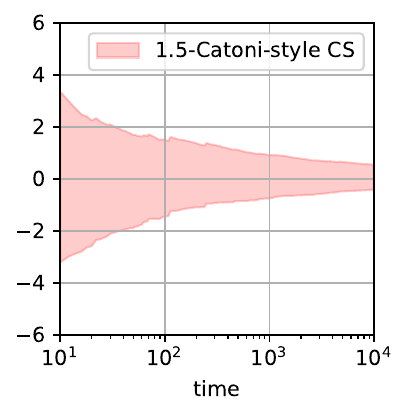}
  \end{minipage}\hfill
  \begin{minipage}[c]{0.6\textwidth}
    \caption{\small $p$-Catoni-style ($p=1.5$) confidence sequence $\{ {\CICT_t}^{p} \}$ under $X_t \overset{\text{i.i.d.}}{\sim} \text{Pareto}(1.8)$ (recentered to $\mu = 0$), $v = 5$, and $\alpha = 0.05$. The sequence $\{ \lambda_t \}$ follows \eqref{eqn:chen-lambda}.
    } \label{fig:catp}
  \end{minipage}
\end{figure}

\begin{theorem}\label{thm:infvar-tight}
    Suppose the coefficients $\{\lambda_t \}$ and the conditional $p$\textsuperscript{th} moments $\{ v_t \}$ are all non-random and let $0 < \varepsilon < 1$. Suppose further that
    \begin{equation}\label{eqn:quadratic-condition-p}
         \left( \sum_{i=1}^t \lambda_i^p \right)^{1/(p-1)} \left( \sum_{i=1}^t  \lambda_i \right)^{-p/(p-1)} \left( \sum_{i=1}^t  5 \lambda_i^p v_i/ p + \log(2/\varepsilon) + \log(2/\alpha) \right) \le \frac{p-1}{p}.
    \end{equation}
    Then, with probability at least $1-\varepsilon$,
    \begin{equation}
        |{\CICT_t}^p | \le \frac{  \sum_{i=1}^t  10 \lambda_i^p v_i + 2p\log(2/\varepsilon) + 2p\log(2/\alpha)  }{\sum_{i=1}^t  \lambda_i}.
    \end{equation}
\end{theorem}

Similar to the case in \cref{thm:catoni-tight}, \eqref{eqn:quadratic-condition-p} is an entirely deterministic inequality when $\{ v_t \}$ and $\{\lambda_t \}$ are all nonrandom. When $v_t = \Otheta(1)$ and $\lambda_t = \Otheta (t^{-1/p})$, which is the case for \eqref{eqn:chen-lambda}, the condition \eqref{eqn:quadratic-condition-p} holds for large $t$ since $\sum_{i=1}^t \lambda_i = \Otheta(t^{(p-1)/p})$ while $\sum_{i=1}^t \lambda_i^p$ grows logarithmically. This gives us the following qualitative version of \cref{thm:infvar-tight}, like (and generalizing) \cref{cor:catoni-tight}.

\begin{corollary}\label{cor:infvar-tight} Suppose $ v_t = \Otheta(1), \lambda_t  = \Otheta (t^{-1/p})$, both nonrandom, such as in~\eqref{eqn:chen-lambda}. Then if $t > \polylog(1/\varepsilon, 1/\alpha)$, with probability at least $1 - \varepsilon$,
\begin{equation}
     |{\CICT_t}^p|\le \Otilde( t^{-(p-1)/p}(\log(1/\varepsilon) + \log(1/\alpha))).
\end{equation}
Here, the notation $\Otilde$ only hides logarithmic factors in $t$.
\end{corollary}
We remark that this shrinkage rate, up to a logarithmic factor in $t$, matches the lower bound for CIs by \citet[Theorem 3.1]{devroye2016}. If we let $\{v_t\}$ grow, say in a rate of $v_t = \Otheta(t^\gamma)$, one may match the scale growth of data by adjusting the $\{ \lambda_t \}$ sequence to a more decreasing one, in order to optimize the width bound in \cref{thm:infvar-tight}.
\begin{corollary}\label{cor:infvar-tight-2} Suppose $0 \le \gamma < p-1$ and $v_t = \Otheta(t^\gamma), \lambda_t  = \Otheta (t^{-(1+\gamma)/p})$, both nonrandom. Then if $t > \polylog(1/\varepsilon, 1/\alpha)$, with probability at least $1 - \varepsilon$,
\begin{equation}
     |{\CICT_t}^p|\le \Otilde( t^{-(p-1 - \gamma)/p }(\log(1/\varepsilon) + \log(1/\alpha))).
\end{equation}
Here, the notation $\Otilde$ only hides logarithmic factors in $t$.
\end{corollary}


\section{Discussions and extensions}

\subsection{Minimality of the moment assumptions}\label{sec:minimality}

We stress here that an upper bound on a $(1+\delta)$\textsuperscript{th} moment, for example the upper variance bound $\sigma^2$ in \cref{ass:var}, is required to be known. We have seen in \cref{sec:inf-var-in-text} that \cref{ass:var} can be \emph{weakened} in various ways, but is not \emph{eliminated} since another moment bound is introduced.

Such assumptions, strong as they may seem at first sight, are necessitated by the results of \citet{bahadur1956nonexistence}, which immediately imply that if no upper bound on a moment is known \emph{a priori}, mean estimation is provably impossible. Indeed, without a known moment bound, even nontrivial tests for whether the mean equals zero do not exist, meaning that all tests have trivial power (since power is bounded by the type-I error), and thus cannot have power going to 1 while the type-I error stays below $\alpha$. The lack of power one tests for a point null, thanks to the
duality
between CIs and families of tests, in turn implies the impossibility of intervals that shrink to zero width. In a similar spirit, one can see that the lower bound of \cref{prop:gaussianlowerbound} grows to infinity as $\sigma$ does, indicating that a confidence interval (hence a confidence sequence) must be unboundedly wide when no bound on $\sigma$ is in place.

\subsection{Drifting means}\label{sec:non-constant}
Our three confidence sequences also extend, at least in theory, to the case when \cref{ass:mean} is weakened to
\begin{equation}
    \forall t\in\mathbb N^+, \quad \Exp[X_t \mid \mathcal{F}_{t-1}] = \mu_t
\end{equation}
where $\{ \mu_t \}$ is any predictable process. This, in conjunction with \cref{sec:inf-var-in-text}, implies that our work provides a unified framework for \emph{any} $L^2$ process $\{X_t\}$. Such generalization is done by replacing every occurrence of $(X_i - \mu)$ in the martingales \eqref{eqn:ds-martingale} and supermartingales \eqref{eqn:sn-nsm}, \eqref{eqn:cat-sm} by $(X_i - \mu_i)$. The closed-form Dubins-Savage confidence sequence \eqref{eqn:cids} now tracks the weighted average
\begin{equation}
    \mu^\star_t = \frac{\sum_{i=1}^t \lambda_i \mu_i}{\sum_{i=1}^t \lambda_i}
\end{equation}
by $\Pr[\forall t \in \mathbb N^+, \mu^\star_t \in \CIDS_t] \ge 1 - \alpha$. In the case of self-normalized and Catoni-style confidence sequence, a \emph{confidence region} $\CR_t \subseteq \mathbb R^t$ can be solved from the Ville's inequality at each $t$, such that $\Pr[ \forall t \in \mathbb N^+, (\mu_1, \dots, \mu_t) \in \CR_t ] \ge 1-\alpha$. The exact geometry of such confidence regions shall be of interest for future work.

\subsection{Sharpening the confidence sequences by a running intersection}
It is easy to verify that if $\{ \CI_t \}$ forms a $(1-\alpha)$-CS for $\mu$, so does the sequence of running intersections 
\begin{equation}
\widetilde{\CI_t} = \bigcap_{i=1}^t \CI_i,    
\end{equation}
a fact first pointed out by \citet{darling1967confidence}. The intersected sequence $\{\widetilde{\CI_t}\}$ is at least at tight as the original one $\{ \CI_t \}$, while still enjoying a same level of sequential confidence. However, \citet[Section 6]{howard2021time} points out that such practice does not extend to the drifting parameter case, and may suffer from empty interval. We remark that, following the discussion in \cref{sec:non-constant}, we can still perform an intersective tightening under drifting means. To wit, if $\{\CR_t \}$ is a confidence region sequence satisfying $\Pr[ \forall t \in \mathbb N^+, (\mu_1, \dots, \mu_t) \in \CR_t ] \ge 1-\alpha$, so is the sequence formed by
\begin{equation}
    \widetilde{\CR_t} = \bigcap_{i=1}^t \left( \CR_i \times \mathbb R^{t-i}\right).
\end{equation}
The peril of running intersection, however, is that it may result in an empty interval. Though this happens with probability less than $\alpha$ by the definition of CS, an empty interval is a problematic result in practice that one would like to avoid.

\subsection{Sequential hypothesis testing and one-sided 
 inexact nulls}\label{sec:testing}

It is clear that our construction of confidence sequences provides a powerful tool to sequentially test the null hypothesis of \cref{ass:mean} against the alternative
\begin{equation}
    H_a : \quad \forall t\in\mathbb N^+, \; \Exp[X_t \mid \mathcal{F}_{t-1}] = \mu_a \neq \mu,
\end{equation}
viz., we reject \cref{ass:mean} \emph{whenever} $\mu \notin \CI_t$. For the self-normalized CS and Catoni-style CS which are based on a pair of supermartingales, it is when either of the supermartingales exceeds $2/\alpha$, and there is no need to explicitly calculate the interval $\CISN_t$ or $\CICT_t$. When these supermartingales are used in this manner, they are commonly called \emph{test supermartingales} for the null \cref{ass:mean}. The case of the Dubins-Savage CS is slightly different, discussed in \cref{sec:ds-discussion}.

Further, it is common in statistics to consider the following one-sided testing problem:
\begin{gather}
  H_0 : \quad \forall t\in\mathbb N^+, \; \Exp[X_t \mid \mathcal{F}_{t-1}] \le \mu; \label{eqn:null}
  \\
      H_a : \quad \forall t\in\mathbb N^+, \; \Exp[X_t \mid \mathcal{F}_{t-1}] > \mu.
\end{gather}
We demonstrate briefly here that our invention can easily handle this scenario. Note that one of the Catoni supermartingales, $\{ \Mcat_t \}$ \eqref{eqn:cat-sm}, is still a supermartingale when the weaker assumption $H_0$ \eqref{eqn:null} instead of \cref{ass:mean} holds. This implies that the following one-sided intervals $\{\CICTplus_t\}$ form a $(1-\alpha)$-confidence sequence\footnote{Even though $\mu$ is no longer a functional of the distribution (but a upper bound thereof), we still call $\{ \CI_t \}$ a $(1-\alpha)$-CS for $\mu$ if \eqref{eqn:cs} holds.} of $\mu$, under any distribution of the null:
\begin{equation}\label{eqn:catoni-cs-onesided}
    \CICTplus_t =  \bigg\{  m \in \mathbb R :   \sum_{i=1}^t  \phi(\lambda_i (X_i  - m ))  \le \frac{\sigma^2 \sum_{i=1}^t \lambda_i^2 }{2} + {\log(1/\alpha)} \bigg\} = [ \min(\CICTplus_t) , \infty).
\end{equation}

Hence, the sequential test of rejecting $H_0$ when $\min(\CICTplus_t) > \mu$ attains type I error control within $\alpha$.

\subsection{Tighter supermartingales}\label{sec:tighter}
If one scrutinizes the proof of our Catoni-style CS (\cref{thm:cs-cat}), from \eqref{eqn:tighter-nsm}, it is tempting to consider the following supermartingales
\begin{gather}
    M^{\mathsf{C}*}_t = \prod_{i=1}^t (1 + \lambda_t(X_i - \mu) + \lambda^2_t (X_i - \mu)^2/2)  \cdot \exp (  - \lambda^2_t \sigma^2/2  ), \label{eqn:tighter-sm1}
    \\
    N^{\mathsf{C}*}_t = \prod_{i=1}^t (1 - \lambda_t(X_i - \mu) + \lambda^2_t (X_i - \mu)^2/2)  \cdot \exp (  - \lambda^2_t \sigma^2/2  ). \label{eqn:tighter-sm2}
\end{gather}
since $\Mcat_t \le M^{\mathsf{C}*}_t$ and $\Ncat_t \le N^{\mathsf{C}*}_t$. By \cref{lem:comparison} in \cref{sec:pf}, this larger pair of supermartingales indeed yields a CS even tighter than the Catoni-style CS. However, we remark that the difference between this tighter CS and the Catoni-style CS is small as $\lambda_t$ decreases; and this tighter CS is computationally infeasible: finding the root of $ M^{\mathsf{C}*}_t, N^{\mathsf{C}*}_t = 2/\alpha$ suffers from non-monotonicity (so that we may not easily find the largest/smallest root which defines the endpoints of the CS) and high sensitivity. However, following the discussion in \cref{sec:testing}, it is easy to \emph{test} if $\mu$ is in this tighter CS, i.e.\ if $ M^{\mathsf{C}*}_t, N^{\mathsf{C}*}_t \le 2/\alpha$ actually holds. Therefore, we recommend that one use the Catoni supermartingales when constructing a CS, but use \eqref{eqn:tighter-sm1}, \eqref{eqn:tighter-sm2} when sequentially testing the null \cref{ass:mean}.

\section{Concluding remarks}

In this paper, we present three kinds of confidence sequences for mean estimation of increasing tightness, under an extremely weak assumption that the conditional variance is bounded. The third of these, the Catoni-style confidence sequence, is shown both empirically and theoretically to be close to the previously known confidence sequences and even confidence intervals that only work under light tails requiring the existence of all moments, as well as their decay.

This elegant result bears profound theoretical implications. We now know that the celebrated rate of shrinkage $\mathcal{O}(t^{-1/2})$ and rate of growth $\mathcal{O}(\log(1/\alpha))$ of confidence intervals produced by MGF-based concentration inequalities (e.g.\ Chernoff bound \eqref{eqn:chernoff}) extend essentially in two directions simultaneously: heavy tail up to the point where only the second moment is required to exist, and sequentialization to the anytime valid regime.

Our work shall also find multiple scenarios of application, many of which are related to multi-armed bandits and reinforcement learning. For example, the best-arm identification problem \citep{jamieson2014best} in the stochastic bandit literature relies on the construction of confidence sequences, and most previous works typically study the cases of Bernoulli and subGaussian bandits. Given the result of this paper, we may now have a satisfactory solution when heavy-tailed rewards \citep{bubeck2013bandits} are to be learned. A similar locus of application is the off-policy evaluation problem \citep{thomas2015high} in contextual bandits, whose link to confidence sequences was recently established \citep{karampatziakis2021off}. While \citet{karampatziakis2021off} only considered bounded rewards, our work provides the theoretical tools to handle a far wider range of instances.

Besides the issues of drifting means we mentioned in \cref{sec:non-constant}, the search for an all-encompassing LIL lower bound we mentioned in \cref{sec:lb}, we also expect future work to address the problem of multivariate or matrix extensions. The study by \citet{catoni2017dimension}, we speculate, can be a starting point. Finally, the online algorithm for approximating the interval in the Catoni-style CS \eqref{eqn:catoni-cs} can also be studied.

\subsubsection*{Acknowledgments}
Research reported in this paper was sponsored in part by the DEVCOM Army Research Laboratory under Cooperative Agreement W911NF-17-2-0196 (ARL IoBT CRA). The views and conclusions contained in this document are those of the authors and should not be interpreted as representing the official policies, either expressed or implied, of the Army Research Laboratory or the U.S.\ Government. The U.S.\ Government is authorized to reproduce and distribute reprints for Government purposes notwithstanding any copyright notation herein.

\bibliography{main}

\newpage

\appendix
\section{Tuning the coefficients $\{\lambda_t\}$}

\subsection{Tuning the coefficients in the Dubins-Savage confidence sequence}\label{sec:ds-tuning}
Note that when \eqref{eqn:const-lambdas} happens, the half-width of the CI at $t^\star$ is
\begin{equation}
   \frac{1}{2}|\CIDS_{t^\star}| = \frac{{2/\alpha - 1} + \sigma^2 t^\star \lambda^{\star 2} }{t^\star \lambda^\star} = \frac{2/\alpha - 1} {t^\star\lambda^\star} + \sigma^2\lambda^\star,
\end{equation}
which obtains optimal width when
\begin{equation}
    \lambda^\star = \sqrt{\frac{2/\alpha - 1}{\sigma^2 t^\star}}.
\end{equation}
With the above guidance, in \cref{thm:cids}, we take
\begin{equation}
    \lambda_t = \sqrt{\frac{2/\alpha - 1}{\sigma^2 t}}.
\end{equation}
Then, the CS half-width at time $t$ is
\begin{equation}
    \frac{1}{2}|\CIDS_{t}| = \frac{ 2/\alpha - 1 + \sigma^2 \sum_{i=1}^t \frac{2/\alpha - 1}{\sigma^2 i}   }{ \sum_{i=1}^t \sqrt{\frac{2/\alpha - 1}{\sigma^2 i}} } =  \frac{ \sqrt{2/\alpha - 1} \, \sigma (1 +  \sum i^{-1})}{\sum i^{-1/2}} \asymp \frac{ \sqrt{2/\alpha - 1} \, \sigma \log t}{\sqrt{t}}.
\end{equation}

\subsection{Tuning the coefficients in the self-normalized confidence sequence}\label{sec:sn-tuning}
Take $t^\star$ as fixed and $\lambda_1 = \lambda_2 = \dots = \lambda^\star$ as constant. Define $S_{\ell} = \sum_{i=1}^{t^\star} X_i^\ell$. The middle interval length is now
\begin{adjustwidth}{-2cm}{-2cm}
 \begin{align*}
    & |\MSN_{t^\star}|
     \\
     =& {\scriptstyle \frac{2{t^{\star}}\lambda^\star - \sqrt{  \left( \frac{\lambda^{\star 2} S_{1}}{3} - {t^{\star}} \lambda^\star \right) ^2 - 4  \frac{t^\star{\lambda^\star} ^2}{6}  \left( \log(2/\alpha) -  {\lambda^\star} S_{1}  + \frac{ {\lambda^\star}^2 (S_{2} + 2\sigma^2 {t^{\star}})}{6} \right)  }  - \sqrt{  \left( \frac{ {\lambda^\star}^2 S_{1}}{3} + {t^{\star}} {\lambda^\star} \right) ^2 - 4  \frac{{t^{\star}} {\lambda^\star} ^2}{6}  \left( \log(2/\alpha) +  {\lambda^\star} S_{1} + \frac{{\lambda^\star}^2 (S_{2} + 2\sigma^2 {t^{\star}})}{6} \right)  } }{\frac{{t^{\star}}{\lambda^\star}^2}{3}} }
     \\
     = & {\scriptstyle \frac{2{t^{\star}} - \sqrt{  \left( \frac{{\lambda^\star} S_{1}}{3} - {t^{\star}} \right) ^2 -   \frac{2{t^{\star}}}{3}  \left( \log(2/\alpha) -  {\lambda^\star} S_{1}  + \frac{ {\lambda^\star}^2 (S_{2} + 2\sigma^2 {t^{\star}})}{6} \right)  }  - \sqrt{  \left( \frac{ {\lambda^\star} S_{1}}{3} + {t^{\star}}  \right) ^2 -  \frac{2{t^{\star}}}{3}  \left( \log(2/\alpha) +  {\lambda^\star} S_{1}  + \frac{{\lambda^\star}^2 (S_{2} + 2\sigma^2 {t^{\star}})}{6} \right)  } }{\frac{{t^{\star}}{\lambda^\star}}{3}}}
     \\
      \approx & \frac{ - \frac{ \frac{{\lambda^\star}^2 (S_{1})^2}{9} - 2{t^{\star}} \frac{{\lambda^\star}S_{1}}{3} -  \frac{2{t^{\star}}}{3}  \left( \log(2/\alpha) -  {\lambda^\star} S_{1}  + \frac{ {\lambda^\star}^2 (S_{2} + 2\sigma^2 {t^{\star}})}{6} \right)  }{2{t^{\star}}} - \frac{ \frac{{\lambda^\star}^2 (S_{1})^2}{9} + 2{t^{\star}} \frac{{\lambda^\star} S_{1}}{3} -  \frac{2{t^{\star}}}{3}  \left( \log(2/\alpha) +  {\lambda^\star} S_{1}  + \frac{ {\lambda^\star}^2 (S_{2} + 2\sigma^2 {t^{\star}})}{6} \right)  }{2{t^{\star}}} }{\frac{{t^{\star}}{\lambda^\star}}{3}}
      \\
      = & \frac{  \frac{4{t^{\star}}}{3}  \left( \log(2/\alpha)   + \frac{ {\lambda^\star}^2 (S_{2} + 2\sigma^2 {t^{\star}})}{6} \right) - \frac{2{\lambda^\star}^2 (S_{1})^2}{9}  }{\frac{2{t^{\star}}^2 {\lambda^\star}}{3}}
      =  \frac{ \frac{2{t^{\star}}(S_{2}+2 \sigma^2 {t^{\star}}) - 2(S_{1})^2}{9}  {\lambda^\star} + \frac{4{t^{\star}} \log(2/\alpha)}{3} \frac{1}{{\lambda^\star}} }{\frac{2{t^{\star}}^2}{3}},
 \end{align*}\end{adjustwidth}
where we use the approximation $\sqrt{t^2 + \text{smaller term}} \approx t + \frac{\text{smaller term}}{2t}$.
Examining the final expression, the optimal ${\lambda^\star}$ is hence taken as
\begin{equation}\label{eqn:sn-opt}
    \lambda^\star \approx \sqrt{\frac{6t^\star \log(2/\alpha)}{t^\star(S_{2}+2 \sigma^2 t^\star) - (S_{1})^2}}.
\end{equation}
Since we need $\lambda_t$ to be $\mathcal{F}_{t-1}$-measurable, we replace $S_{2}$ and $S_{1}$ with $\sum_{i=1}^{t-1} X^2_i $ and $\sum_{i=1}^{t-1} X_i$, but all other occurrences of $t^\star$ with $t$ in \eqref{eqn:sn-opt} to obtain our predictable sequence $\{\lambda_t\}$ of choice for \cref{thm:cisn},
\begin{equation}
    \lambda_t = \sqrt{\frac{6t \log(2/\alpha)}{t(\sum_{i=1}^{t-1} X_i^2 +2 \sigma^2 t) - (\sum_{i=1}^{t-1} X_i)^2}}.
\end{equation}

\section{Discussion on the Dubins-Savage confidence sequence}
\label{sec:ds-discussion}

We first present here a short and self-contained proof of the Dubins-Savage inequality \citep{dubins1965tchebycheff,khan2009p}.

\begin{proof}[Proof of \cref{lem:ds}] 
Consider the function $Q(x) = \frac{1}{1-\min(x,0)}$. It is not hard to see that, for any $x \in \mathbb R$ and $m\le0$,
\begin{equation}\label{eqn:Q-upperbound}
  Q(x) \le \frac{ (1-m) + (x-m) + (x-m)^2 }{(1-m)^2};
\end{equation}
and for any $x,m\le 0$,
\begin{equation}\label{eqn:Q-lowerbound}
     \frac{ 1-2m-x }{(1-m-x)^2} \le Q(x).
\end{equation}
See \cref{fig:dsfunc} for an illustration.

\begin{figure}[!b]
    \centering
    \includegraphics[width=0.45\textwidth]{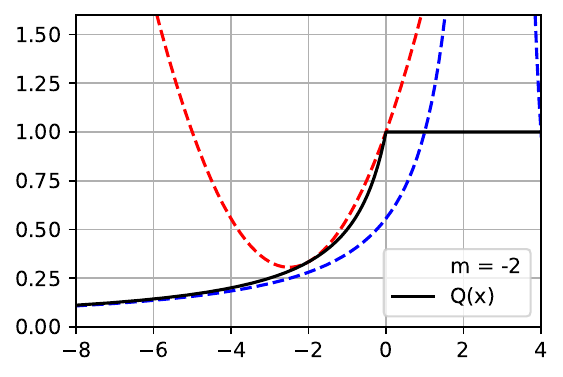} \includegraphics[width=0.45\textwidth]{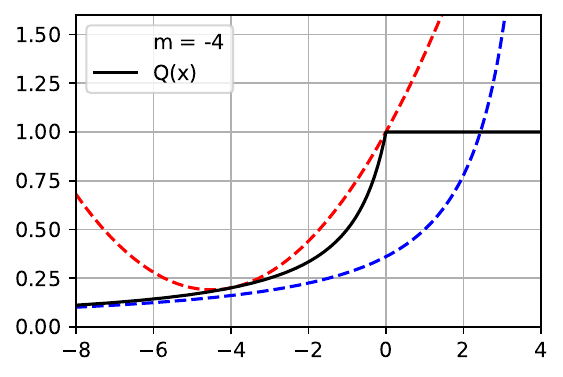}
    \caption{\small Illustration of the upper (\eqref{eqn:Q-upperbound}, plotted in dashed red) and lower (\eqref{eqn:Q-lowerbound}, plotted in dashed blue) bounding functions of $Q(x)$ with $m=-2$, $-4$.}
    \label{fig:dsfunc}
\end{figure}
Now define the following random variables:
\begin{gather}
    x_t =  - ab + bM_t - b^2 \sum_{i=1}^t V_i; \\
    m_{t} = \Exp[ x_t  |\mathcal{F}_{t-1}] =  -ab + bM_{t-1}- b^2 \sum_{i=1}^{t} V_{i}; \\
    m_t' = m_t - x_{t-1}  = -b^2 V_t .
\end{gather}

We shall show that $\{Q(x_t)\}$ is a supermartingale, meaning that $\Exp[ Q(x_t) | \mathcal{F}_{t-1} ] \le Q(x_{t-1})$. First, if $x_{t-1} > 0$, the inequality is trivial. Second, when $x_{t-1} \le 0$, we have $m_t = x_{t-1} + m'_t \le 0.$ So, using \eqref{eqn:Q-upperbound} and \eqref{eqn:Q-lowerbound},
\begin{align}
    & \Exp[ Q(x_t) | \mathcal{F}_{t-1} ] \le \Exp\left[  \frac{ (1-m_{t}) + (x_t-m_{t}) + (x_t-m_{t})^2 }{(1-m_t)^2} \middle \vert \mathcal{F}_{t-1} \right]
    \\
    = & \Exp \left[ \frac{(1-m_{t}) + b(M_t-M_{t-1}) +  b^2(M_t-M_{t-1})^2 }{(1-m_{t})^2} \middle \vert \mathcal{F}_{t-1} \right]
    \\
    =&  \frac{(1-m_{t}) +  b^2 V_t^2 }{(1-m_{t})^2} = \frac{ 1-2m_t'-x_{t-1} }{(1-m_t'-x_{t-1})^2} \le Q( x_{t-1}).
\end{align}
Since $x_0=-ab$ and $Q(x_0)=1/(1+ab)$, we define $R(x) := (1+ab)Q(x)$ to obtain a nonnegative supermartingale $\{R(x_t)\}$ with $R(x_0) = 1$, on which we can use Ville's inequality (\cref{lem:vi}) to conclude that
    \begin{align}
        \Pr\left[  \exists t \in \mathbb N^+, \; M_t \ge a + b \sum_{i=1}^t V_i \right] &= \Pr\left[  \exists t \in \mathbb N^+, \; x_t \ge 0 \right] = \Pr\left[  \exists t \in \mathbb N^+, \; Q(x_t) \ge 1 \right] \\ &= \Pr[ \exists t \in \mathbb N^+, \; R(x_t) \ge 1+ab ] \le \frac{1}{1+ab},
    \end{align}
    concluding the proof.
\end{proof}

Indeed, we can see from the proof that the Dubins-Savage inequality can actually be derived from Ville's inequality on a nonnegative supermartingale. In the parlance of \cref{sec:testing}, 
the process $\{ R(x_t) \}$ can be used as a test supermartingale for the null \cref{ass:mean}, when setting $M_t $ to be $ \sum \lambda_i(X_i - \mu)$.
However, there is a major difference in how this test supermartingale relates to the
the Dubins-Savage confidence sequence: if one fixes \emph{a priori} the parameters $a$ and $b$, the rejection rule $R(x_t) \ge 1/\alpha$ is equivalent to the Dubins-Savage CS (\cref{thm:cids}) \emph{only when $\alpha = 1/(1+ab)$}. This is unlike the cases of the other two CSs in this paper, where the duality between confidence sequence and sequential testing  holds for any $\alpha$.


\section{Omitted proofs and additional propositions}\label{sec:pf}

\begin{proof}[Proof of \cref{thm:cids}] 
We apply \cref{lem:ds} to the two martingales \eqref{eqn:ds-martingale} with $a = (2/\alpha - 1)/b$. Then we have,
\begin{gather}\label{eqn:d-b-cs-pre}
    \Pr\left[\forall t \in \mathbb N^+, \; \sum_{i=1}^t \lambda_i (X_i -\mu) \le \frac{2/\alpha - 1}{b} + b \sum_{i=1}^t \lambda_i^2 \, \Exp[(X_i -\mu)^2  \vert \mathcal{F}_{i-1}] \right] \ge 1 - \frac{\alpha}{2};
    \\
    \Pr\left[\forall t \in \mathbb N^+, \; -\sum_{i=1}^t \lambda_i (X_i -\mu) \le \frac{2/\alpha - 1}{b} + b \sum_{i=1}^t \lambda_i^2 \, \Exp[(X_i -\mu)^2  \vert \mathcal{F}_{i-1}] \right] \ge 1 - \frac{\alpha}{2}.
\end{gather}
Using \cref{ass:var}, we then have,
\begin{gather}\label{eqn:d-b-cs}
    \Pr\left[\forall t \in \mathbb N^+, \; \sum_{i=1}^t \lambda_i (X_i -\mu) \le \frac{2/\alpha - 1}{b} + b \sigma^2 \sum_{i=1}^t \lambda_i^2 \right] \ge 1 - \frac{\alpha}{2};
    \\
    \Pr\left[\forall t \in \mathbb N^+, \; -\sum_{i=1}^t \lambda_i (X_i -\mu) \le \frac{2/\alpha - 1}{b} + b \sigma^2 \sum_{i=1}^t \lambda_i^2 \right] \ge 1 - \frac{\alpha}{2}.
\end{gather}
We remark here that the parameter $b$ in the inequalities above is actually redundant and can be eliminated (i.e., take $b = 1$), since tuning $b$ is equivalent to tuning the coefficients $\lambda_i$. To wit, multiplying $b$ by a constant $\lambda_0$ results in the same inequalities as dividing \emph{each} $\lambda_i$ by $\lambda_0$.

Putting $b = 1$ in the inequalities above and taking a union bound, we immediately arrive at the result.
\end{proof}

\begin{proof}[Proof of \cref{lem:sn-martingale}] It is known (see \citet[Proposition 12]{delyon2009exponential}) that $\exp(x - x^2/6) \le 1 + x + x^2/3$ for all $x \in \mathbb R$. Therefore,
    \begin{align*}
    & \Exp\left[\exp\left(\lambda_t (X_t - \mu) - \frac{\lambda^2_t ( (X_t - \mu)^2 + 2 \sigma^2 )}{6} \right) \middle \vert \mathcal{F}_{t-1} \right]
    \\
    = \, & \Exp\left[ \exp\left( \lambda_t (X_t - \mu) - \frac{\lambda^2_t  (X_t - \mu)^2 }{6} \right) \middle \vert \mathcal{F}_{t-1} \right] \cdot \exp\left(  - \frac{\lambda^2_t \sigma^2 }{ 3} \right)
    \\
    \le \, & \Exp\left[ 1 + \lambda_t(X_t - \mu) + \frac{\lambda^2_t}{3}(X_t - \mu)^2 \middle \vert \mathcal{F}_{t-1} \right] \cdot \exp\left(  - \frac{\lambda^2_t \sigma^2 }{ 3} \right)
    \\
    \le \, & \left( 1 + \frac{\lambda^2_t \sigma^2}{3} \right)  \cdot \exp\left(  - \frac{\lambda^2_t \sigma^2 }{ 3} \right) \le 1.
\end{align*}
It immediately follows from \cref{lem:mult-nsm} that $\{ \Msn_t\}$ is a nonnegative supermartingale.
\end{proof}

\begin{proof}[Proof of \cref{lem:anti-cisn}]
    Applying Ville's inequality (\cref{lem:vi}) on the nonnegative supermartingale \eqref{eqn:sn-nsm}, we have that
\begin{equation}
    \Pr\left[  \forall t \in \mathbb N^+, \; \sum_{i=1}^t \left(\lambda_t (X_t - \mu) - \frac{\lambda^2_t ( (X_t - \mu)^2 + 2 \sigma^2 )}{6} \right) \le \log(2/\alpha)  \right] \ge 1 - \frac{\alpha}{2}.
\end{equation}
Solving $\mu$ from the quadratic inequality
\begin{equation}
    \sum_{i=1}^t \left(\lambda_t (X_t - \mu) - \frac{\lambda^2_t ( (X_t - \mu)^2 + 2 \sigma^2 )}{6} \right) > \log(2/\alpha)
\end{equation}
yields the interval (each $\sum$ standing for $\sum_{i=1}^t$)
\begin{equation}\label{eqn:acisn+}
{
     \left[ \frac{ \left( \frac{\sum \lambda_i^2 X_i}{3} - \sum \lambda_i \right) \pm \sqrt{  \left( \frac{\sum \lambda_i^2 X_i}{3} - \sum \lambda_i \right) ^2 - 4  \frac{\sum \lambda_i ^2}{6}  \left( \log(2/\alpha) - \sum \lambda_i X_i  + \frac{\sum \lambda_i^2 (X_i^2 + 2\sigma^2)}{6} \right)  }    }{ \frac{\sum \lambda_i ^2}{3} }  \right],
}
\end{equation}
which then forms a $(1-\alpha/2)$-anticonfidence sequence for $\mu$, the $\aCISNp_{t}$ at issue. Another $(1-\alpha/2)$-anticonfidence sequence, $\aCISNn_{t}$, can be formed by replacing each $\lambda_t$ with $-\lambda_t$,
{
\begin{equation}\label{eqn:acisn-}
    \left[  \frac{ \left( \frac{\sum \lambda_i^2 X_i}{3} + \sum \lambda_i \right) \pm \sqrt{  \left( \frac{\sum \lambda_i^2 X_i}{3} + \sum \lambda_i \right) ^2 - 4  \frac{\sum \lambda_i ^2}{6}  \left( \log(2/\alpha) + \sum \lambda_i X_i  + \frac{\sum \lambda_i^2 (X_i^2 + 2\sigma^2)}{6} \right)  }    }{ \frac{\sum \lambda_i ^2}{3} } \right].
\end{equation}
}

\end{proof}

\begin{proof}[Proof of \cref{lem:cat-sm}] We observe that
    \begin{align}
    & \Exp\left[\exp \left( \pm \phi(\lambda_t (X_t  -\mu)) - \lambda^2_t \sigma^2/2 \right) \middle \vert \mathcal{F}_{t-1} \right]
    \\
    \le \, & \Exp\left[ 1 \pm \lambda_t(X_t - \mu) + \lambda^2_t (X_t - \mu)^2/2 \middle \vert \mathcal{F}_{t-1}  \right] \cdot \exp\left(  - \lambda^2_t \sigma^2/2  \right) \label{eqn:tighter-nsm}
    \\
    \le \, & (1 + \lambda^2_t \sigma^2/2) \exp(-\lambda^2_t \sigma^2/2) \le 1.
\end{align}
Hence $\{ \Mcat_t \}$ and $\{ \Ncat_t \}$ are both nonnegative supermartingales by \cref{lem:mult-nsm}.
\end{proof}

We have the following statement on the tightness of \cref{lem:cat-sm}, which states that the variance bound \cref{ass:var} is necessary for the processes $\{ \Mcat_t \}$ and $\{ \Ncat_t \}$ to be supermartingales --- the violation of \cref{ass:var} on any non-null set will prevent $\{ \Mcat_t \}$ and $\{ \Ncat_t \}$ from being supermartingales.
\begin{proposition}\label{prop:tightness-cat-nsm}
    Let $\{ Y_t \}$
 be a
 process that satisfies \cref{ass:mean}, but there exists a $t$ such that $\Exp[ (Y_t - \mu)^2 | \mathcal{F}_{t-1} ] = v_{t}$ and $S = \{ v_t > (1+2\kappa)\sigma^2 \}$ is a set in $ \mathcal{F}_{t-1}$ with non-zero $\Pr$-measure, where $\kappa > 0$. Further, suppose there exists a $\delta > 0$ such that $\Exp[ |Y_t - \mu|^{2+\delta} | \mathcal{F}_{t-1} ] < \infty$. Then, there exists a $\mathcal{F}_{t-1}$-measurable $\lambda_t$, such that
 \begin{equation}
    \Exp \left[\exp \left( \pm \phi(\lambda_t (X_t  -\mu)) - \lambda^2_t \sigma^2/2 \right) \middle \vert \mathcal{F}_{t-1} \right] > 1
 \end{equation}
 $\Pr$-a.s.\ on $S$.
\end{proposition}
\begin{proof}[Proof of \cref{prop:tightness-cat-nsm}]
Let $\eta$ be any real number in $(0, 1/2)$. There exists a positive number $x_b$ such that for any $\phi$ that is a Catoni-type influence function, when $|x| < x_\eta$, $\phi(x) \ge \log(1 + x + \eta x^2)$. Note that
\begin{equation}
    \Exp\left[\exp \left(\phi(\lambda (Y_t  -\mu)) \right) | \mathcal{F}_{t-1} \right] \le 1 + \frac{1}{2}\lambda^2 v_t.
\end{equation}
And we have
\begin{align}
    &\liminf_{\lambda \to 0} \lambda^{-2}\left( \Exp\left[\exp \left(\phi(\lambda (Y_t  -\mu)) \right)  |\mathcal{F}_{t-1} \right] - 1 \right)
    \\
    \ge & \liminf_{\lambda \to 0} \lambda^{-2}\left( \Exp\left[ \id_{\{|\lambda(Y_t - \mu)| < x_\eta\}} \exp \left(\phi(\lambda (Y_t  -\mu)) \right)  |\mathcal{F}_{t-1} \right] -1 \right)
    \\
    \ge & \liminf_{\lambda \to 0}  \lambda^{-2}\left( \Exp \left[\id_{\{ |\lambda(Y_t - \mu)| < x_\eta \}} (1 + \lambda(Y - \mu) +  \eta\lambda^2 (Y_t - \mu)^2)  |\mathcal{F}_{t-1} \right] -1 \right)
    \\
    \ge & \liminf_{\lambda \to 0}  \lambda^{-2}( \Pr[ |\lambda(Y_t - \mu)| < x_\eta  |\mathcal{F}_{t-1} ] - 1) + \liminf_{\lambda \to 0} \lambda^{-2} \Exp \left[\id_{\{ |\lambda(Y_t - \mu)| < x_\eta \}} \lambda(Y_t - \mu)  |\mathcal{F}_{t-1} \right] + \\ & \liminf_{\lambda \to 0} \lambda^{-2} \Exp \left[\id_{\{ |\lambda(Y_t - \mu)| < x_\eta \}}\eta\lambda^2 (Y_t - \mu)^2   |\mathcal{F}_{t-1} \right] 
    \\
    = & 0 + 0 + \eta v_{t} = \eta  v_{t} .
\end{align}
The first two limits inferior above are 0 since $Y_t$ has finite conditional $(2+\delta)$\textsuperscript{th} moment.

Since $\eta$ is arbitrary in $(0, 1/2)$ it follows that
\begin{equation}\label{eqn:liminf-lambda}
\liminf_{\lambda \to 0} \lambda^{-2}\left( \Exp\left[\exp \left(\phi(\lambda (Y_t  -\mu)) \right)  | \mathcal{F}_{t-1}\right] - 1 \right) \ge  \frac{v_t}{2}.    
\end{equation}
(Actually it is not hard to see that the inequality above is equality.) Now, recall that $v_t > (1+2\kappa) \sigma^2$ on the set $S \in \mathcal{F}_{t-1}$. By \eqref{eqn:liminf-lambda} there exists some $\lambda_0$ such that when $\lambda < \lambda_0$, $\Exp\left[\exp \left(\phi(\lambda (Y_t  -\mu)) \right) | \mathcal{F}_{t-1} \right] \ge 1 + \frac{1}{2}\lambda^2 (1+\kappa)\sigma^2$ on $S$.

Let $g(\kappa)$\footnote{One may express $g$ by the Lambert $W$ function.} be the unique positive zero of $\e^x - 1-(1+\kappa)x$; so when $x \in (0, g(\kappa))$, $1+(1+\kappa)x > \e^{x}$.  Hence, when $0 < \lambda \le \min \{ \lambda_0, \sqrt{2g(\kappa)}/\sigma \}$,
\begin{equation}
    \Exp\left[\exp \left(\phi(\lambda (Y_t  -\mu)) \right) | \mathcal{F}_{t-1} \right]  \ge  1 + \frac{1}{2}\lambda^2 (1+\kappa)\sigma^2 >\exp\left(\frac{1}{2}\lambda^2 \sigma^2\right)
\end{equation}
on $S$. The case for $ \Exp\left[\exp \left(-\phi(\lambda (Y_t  -\mu)) \right) | \mathcal{F}_{t-1} \right]$ is analogous.
 \end{proof}

\begin{proof}[Proof of \cref{thm:catoni-tight}]
We define $f_t(m)$ to be the random function
\begin{equation}
    f_t(m) = \sum_{i=1}^t  \phi(\lambda_i (X_i  - m ))
\end{equation}
in \eqref{eqn:catoni-cs}, which is always strictly decreasing in $m$.
First, for all $m \in \mathbb R$, let
\begin{equation}
M_t(m) = \prod_{i=1}^t \exp \left( \phi(\lambda_i (X_i - m)) - \lambda_i (\mu - m) - \frac{\lambda_i ^ 2}{2}(\sigma^2 + (\mu - m)^2))  \right). 
\end{equation}
Let $v_t = \Exp[(X_t - \mu)^2 \vert \mathcal{F}_{t-1}]$. Note that
\begin{align*}
    & \Exp \left[ \exp \left( \phi(\lambda_t (X_t - m)) - \lambda_t (\mu - m) - \frac{\lambda_t ^ 2}{2}(\sigma^2 + (\mu - m)^2)  \right) \middle \vert \mathcal{F}_{t-1} \right]
    \\
    = & \Exp\left[ \exp (\phi(\lambda_t (X_t - m))) \middle \vert \mathcal{F}_{t-1} \right ] \exp\left(- \lambda_t (\mu - m)  - \frac{\lambda_t ^ 2}{2}(\sigma^2 + (\mu - m)^2) \right)
    \\
    \le & \Exp\left[  1 + \lambda_t (X_t - m) + \frac{\lambda_t^2}{2}(X_t - m)^2   \middle \vert \mathcal{F}_{t-1} \right ] \exp\left(- \lambda_t (\mu - m)  - \frac{\lambda_t ^ 2}{2}(\sigma^2 + (\mu - m)^2) \right)
    \\
    = & \left( 1 + \lambda_t(\mu - m) + \frac{\lambda_t^2  }{2}(v_t + (\mu - m)^2) \right) \exp\left(- \lambda_t (\mu - m)  - \frac{\lambda_t ^ 2}{2}(\sigma^2 + (\mu - m)^2) \right)
    \\
    \le & \left( 1 + \lambda_t(\mu - m) + \frac{\lambda_t^2  }{2}(\sigma^2+ (\mu - m)^2) \right) \exp\left(- \lambda_t (\mu - m)  - \frac{\lambda_t ^ 2}{2}(\sigma^2 + (\mu - m)^2) \right)
    \\
    \le & \ 1.
\end{align*}
Hence, again due to \cref{lem:mult-nsm}, $\{ M_t(m) \}$ is a nonnegative supermartingale. Note that $\{ M_t(\mu) \}$ is just the Catoni supermartingale $\{ \Mcat_t \}$ defined in \eqref{eqn:cat-sm}.
We hence have $\Exp M_t(m) \le 1$; that is,
\begin{equation}
    \Exp \exp(f_t(m)) \le \exp \left( \sum_{i=1}^t \left(  \lambda_i (\mu - m) + \frac{\lambda_i^2}{2}(\sigma^2 + (\mu - m)^2) \right) \right).
\end{equation}
Similarly, define $N_t (m) = \prod_{i=1}^t \exp \left( - \phi(\lambda_i (X_i - m)) + \lambda_i (\mu - m) - \frac{\lambda_i ^ 2}{2}(\sigma^2 + (\mu - m)^2))  \right)$ and it will also be a nonnegative supermartingale for all $m \in \mathbb R$, with
\begin{equation}
     \Exp \exp(-f_t(m)) \le \exp \left( \sum_{i=1}^t \left( - \lambda_i (\mu - m) + \frac{\lambda_i^2}{2}(\sigma^2 + (\mu - m)^2) \right) \right).
\end{equation}

Define the functions
\begin{align}
     B^+_t(m) = & \sum_{i=1}^t \left(  \lambda_i (\mu - m) + \frac{\lambda_i^2}{2}(\sigma^2 + (\mu - m)^2) \right) + \log(2/\varepsilon)
     \\
     = & \sum_{i=1}^t \frac{\lambda_i^2}{2}(m - \mu)^2 - \sum_{i=1}^t \lambda_i (m - \mu) + \sum_{i=1}^t \frac{\lambda_i^2 \sigma^2}{2} + \log(2/\varepsilon);
     \\
     B^-_t(m) = & \sum_{i=1}^t \left(  \lambda_i (\mu - m) - \frac{\lambda_i^2}{2}(\sigma^2 + (\mu - m)^2) \right) - \log(2/\varepsilon)
     \\
     = & -\sum_{i=1}^t \frac{\lambda_i^2}{2}(m - \mu)^2 - \sum_{i=1}^t \lambda_i (m - \mu) - \sum_{i=1}^t \frac{\lambda_i^2 \sigma^2}{2} - \log(2/\varepsilon).
\end{align}
By Markov's inequality,
\begin{gather}
     \forall m \in \mathbb R, \; \Pr[ f_t (m) \le B^+_t(m) ] \ge 1- \varepsilon/2, \label{eqn:func-markov-bound}
     \\
      \forall m \in \mathbb R, \; \Pr[ f_t (m) \ge B^-_t(m) ] \ge 1- \varepsilon/2.
\end{gather}
Let us now consider the equation
\begin{equation}\label{eqn:qd-pi-t}
    B^+_t(m) =  - \sum_{i=1}^t \frac{\lambda_i^2 \sigma^2}{2} - {\log(2/\alpha)},
\end{equation}
which, by rearrangement, can be written as (each $\sum$ standing for $\sum_{i=1}^t$)
\begin{equation}
    \left(\sum \lambda_i^2 /2 \right)(m - \mu)^2 - \left(\sum \lambda_i \right) (m - \mu) + \left(\sum \lambda_i^2 \sigma^2 + \log(2/\varepsilon) + \log(2/\alpha)\right) = 0.
\end{equation}
As a quadratic equation, it has solutions if and only if 
\begin{equation}
     \left(\sum \lambda_i \right)^2 - 4\left(\sum \lambda_i^2 /2 \right) \left(\sum \lambda_i^2 \sigma^2 + \log(2/\varepsilon) + \log(2/\alpha)\right) \ge 0,
\end{equation}
which is just the condition \eqref{eqn:quadratic-condition}.
Let $m=\pi_t$ be the smaller solution of \eqref{eqn:qd-pi-t}. Since $\{\lambda_t\}$ is assumed to be non-random, the quantity $\pi_t$ is also non-random. Then, we can put $m=\pi_t$ into \eqref{eqn:func-markov-bound},
\begin{equation}
   \Pr\left[ f_t (\pi_t) \le - \sum_{i=1}^t \frac{\lambda_i^2 \sigma^2}{2} - {\log(2/\alpha)} \right] \ge 1- \varepsilon/2.
\end{equation}
Note that $f_t(\max(\CICT_t)) = - \sum_{i=1}^t \frac{\lambda_i^2 \sigma^2}{2} - {\log(2/\alpha)}$. Hence
\begin{equation}
     \Pr[ f_t(\pi_t) \le f_t(\max(\CICT_t))  ] \ge 1- \varepsilon/2,
\end{equation}
which indicates that
\begin{equation}\label{eqn:up-cict-t-bound}
    \Pr[ \max(\CICT_t) \le \pi_t ] \ge 1- \varepsilon/2.
\end{equation}

Now notice that
\begin{align}
   \pi_t  &= \mu + \frac{ \sum \lambda_i - \sqrt{ (\sum \lambda_i)^2 - 4(\sum \lambda_i^2 /2)(\sum \lambda_i^2 \sigma^2 + \log(2/\varepsilon) + \log(2/\alpha))   } }{2 \sum \lambda_i^2 /2}
    \\
    & =  \mu + \frac{ \sum \lambda_i  - (\sum \lambda_i)\sqrt{1 - \frac{4(\sum \lambda_i^2 /2)(\sum \lambda_i^2 \sigma^2 + \log(2/\varepsilon) + \log(2/\alpha))}{(\sum \lambda_i)^2 } }  }{2 \sum \lambda_i^2 /2}  
    \\
    & \le \mu + \frac{ \sum \lambda_i  - (\sum \lambda_i)(1 - \frac{4(\sum \lambda_i^2 /2)(\sum \lambda_i^2 \sigma^2 + \log(2/\varepsilon) + \log(2/\alpha))}{(\sum \lambda_i)^2 } )  }{2 \sum \lambda_i^2 /2}  
    \\
    & = \mu + \frac{ 2(\sum \lambda_i^2 \sigma^2 + \log(2/\varepsilon) + \log(2/\alpha)) }{\sum \lambda_i }. \label{eqn:pi-t-bound}
\end{align}
Combining \eqref{eqn:up-cict-t-bound} and \eqref{eqn:pi-t-bound} gives us the one-sided concentration
\begin{equation}\label{eqn:one-sided}
    \Pr\left[ \max(\CICT_t) \le  \mu + \frac{ 2(\sum \lambda_i^2 \sigma^2 + \log(2/\varepsilon) + \log(2/\alpha)) }{\sum \lambda_i } \right] \ge 1 - \varepsilon / 2.
\end{equation}

Now, let $\rho_t$ be the larger solution of
\begin{equation}
     B^-_t(m) =  \sum_{i=1}^t \frac{\lambda_i^2 \sigma^2}{2} + {\log(2/\alpha)}.
\end{equation}
A similar analysis yields
\begin{equation}
     \Pr[ \min(\CICT_t) \ge \rho_t ] \ge 1- \varepsilon/2,
\end{equation}
and
\begin{align}
    \rho_t &= \mu + \frac{ -\sum \lambda_i + \sqrt{(\sum \lambda_i)^2  - 4(\sum \lambda_i^2 /2)(\sum \lambda_i^2 \sigma^2 + \log(2/\varepsilon) + \log(2/\alpha))  } }{2\sum \lambda_i^2 / 2}
    \\
    & \ge  \mu  - \frac{ 2(\sum \lambda_i^2 \sigma^2 + \log(2/\varepsilon) + \log(2/\alpha)) }{\sum \lambda_i }.
\end{align}
Hence we have the other one-sided concentration
\begin{equation}\label{eqn:other-one-sided}
    \Pr\left[ \min(\CICT_t) \ge  \mu - \frac{ 2(\sum \lambda_i^2 \sigma^2 + \log(2/\varepsilon) + \log(2/\alpha)) }{\sum \lambda_i }\right] \ge 1 - \varepsilon / 2.
\end{equation}

Hence, a union bound on \eqref{eqn:one-sided} and \eqref{eqn:other-one-sided} gives rise to the concentration on the interval width we desire,
\begin{equation}
    \Pr\left[ \max(\CICT_t) - \min(\CICT_t) \le \frac{ 4(\sum \lambda_i^2 \sigma^2 + \log(2/\varepsilon) + \log(2/\alpha)) }{\sum \lambda_i }\right] \ge 1 - \varepsilon.
\end{equation}
This concludes the proof.
\end{proof}

Before we prove \cref{cor:loglogt}, we review the technique of stitching as appeared in \citet[Section 3.1]{howard2021time}. Let $\{ Y_t \}$ be an i.i.d.\ sequence of random variable of mean $\mu$ and subGaussian with variance factor 1. Then, for any $\lambda \in \mathbb R$, the following process is a nonnegative supermartingale, 
\begin{equation}
    \exp\left\{\Lambda (Y_1 + \dots + Y_t - t\mu) - \Lambda^2 t/2 \right \},
\end{equation}
which, in conjunction with Ville's inequality, yields the following ``linear boundary'' confidence sequence,
\begin{equation}
 \mathrm{CI}_t^{\mathsf{lin}}(\Lambda, \alpha) = \left[\widehat{\mu}_t \pm \frac{t \Lambda^2 /2 + \log(2/\alpha)}{t\Lambda} \right].
\end{equation}
The idea of \cite{howard2021time} is to divide $\alpha = \sum_{j=0}^\infty \alpha_j$, take some sequences $\{\Lambda_j\}$ and $\{ t_j \}$ ($t_0=1$), and consider the following CS:
\begin{equation}\label{eqn:stitched-def}
    \CI^{\mathsf{stch}}_t =  \mathrm{CI}_t^{\mathsf{lin}}(\Lambda_j, \alpha_j), \quad \text{when } t_j \le t < t_{j+
    1},
\end{equation}
which is indeed a $(1-\alpha)$-CS due to union bound. \cite{howard2021time} shows that using \emph{geometrically spaced} epochs $\{t_j\}$, the lower bound $\sqrt{\log \log t/ t}$ of the law of the iterated logarithm can be matched. We prove a slightly different bound than \citet[Equation (11)]{howard2021time} below.

\begin{lemma}\label{prop:stitch} Let $t_j = \e^j$, $\alpha_j = \frac{\alpha}{(j+2)^2}$, and $\Lambda_j = \sqrt{ \log(2/\alpha_j)  \e^{ - j} }$. Note that $\sum_{j=0}^\infty \alpha_j < \alpha$.
Then,
\begin{equation}
    \frac{t \Lambda^2_j /2 + \log(2/\alpha_j)}{t\Lambda_j}  \le (1+\e)\sqrt{ \frac{\log(2/\alpha) + 2\log\log \e^2 t}{t} } , \quad \text{when } t_j \le t < t_{j+
    1}.
\end{equation}
This implies that the stitched CS \eqref{eqn:stitched-def} enjoys the following closed-form expression,
\begin{equation}
    \CI^{\mathsf{stch}}_t \subseteq \left[ \widehat{\mu}_t \pm   (1+\e)\sqrt{ \frac{\log(2/\alpha) + 2\log\log \e^2 t}{t} }  \right].
\end{equation}
\end{lemma}
\begin{proof}[Proof of \cref{prop:stitch}] We have
\begin{align}
    &\frac{t \Lambda^2_j /2 + \log(2/\alpha_j)}{t\Lambda_j}  = \frac{ \log(2/\alpha_j) (t\e^{-j} +1)  }{t \sqrt{\log(2/\alpha_j) \e^{-j} } } \le \frac{ \log(2/\alpha_j) (\e +1)  }{t \sqrt{\log(2/\alpha_j) t^{-1} } } = (\e + 1) \sqrt{ \frac{\log(2/\alpha_j)}{t} } 
    \\
    = &  (\e + 1) \sqrt{ \frac{\log(2/\alpha) + 2\log(j+2)}{t} }  \le (\e + 1) \sqrt{ \frac{\log(2/\alpha) + 2\log(\log t+2)}{t} }.
\end{align}
\end{proof}

This ingredient of stitching leads to our \cref{cor:loglogt}.
\begin{proof}[Proof of \cref{cor:loglogt}] 
Applying \cref{thm:catoni-tight}, \eqref{eqn:one-sided} and \eqref{eqn:other-one-sided} to the case where $\lambda_t = \Lambda$, we see that as long as 
\begin{equation}\label{eqn:loglog-condition-t}
    \left( \frac{1}{2} - \Lambda^2 \right) t \ge \log(2/\varepsilon) + \log(2/\alpha),
\end{equation}
we have that
\begin{equation}
    \CICT_t(\Lambda, \alpha) \subseteq \left[ \mu \pm \frac{2(t \Lambda^2 + \log(2/\alpha) + \log(2/\varepsilon))}{t\Lambda} \right] \quad \text{with probability at least } 1- \varepsilon.
\end{equation}
Using the same $t_j = \e^j$, $\alpha_j = \frac{\alpha}{(j+2)^2}$, and $\Lambda_j = \sqrt{ \log(2/\alpha_j)  \e^{ - j} }$ as \cref{prop:stitch},
we see that (taking $\varepsilon = \alpha_0 = \alpha/4$)
\begin{equation}
   \frac{2(t \Lambda^2_j + \log(2/\alpha_j) + \log(2/\varepsilon))}{t\Lambda_j} \le \frac{4(t \Lambda^2_j /2 + \log(2/\alpha_j))}{t\Lambda_j} \le 4(\e + 1) \sqrt{ \frac{\log(2/\alpha) + 2\log\log\e^2 t}{t} } 
\end{equation}
when $t_j \le t < t_{j+
    1}$.
    
Further, we see that
\begin{equation}
    \left( \frac{1}{2} - \Lambda^2_j \right) t_j \ge \log(2/\varepsilon) + \log(2/\alpha_j)
\end{equation}
holds as long as $\e^j > \polylog(1/\alpha)$. Hence \eqref{eqn:loglog-condition-t} is met when $t > \polylog(1/\alpha)$. Combining all above we arrive at the desired conclusion.
\end{proof}

\begin{proof}[Proof of \cref{prop:gaussianlowerbound}]
Let $\htheta_t = \frac{\lw_t + \up_t}{2}$ and $r_t = \frac{w^{Q,\varepsilon}}{2}$. Due to union bound, for any $Q \in \mathcal{Q}_{\sigma^2}$,
\begin{equation}
    \Prw_{\{X_i\} \sim Q} \left[  \htheta_t \le \mu_Q - r_t  \right] \le \frac{\alpha}{2} + \varepsilon, \quad  \Prw_{\{X_i\} \sim Q} \left[  \htheta_t \ge \mu_Q + r_t  \right] \le \frac{\alpha}{2} + \varepsilon.
\end{equation}
Now by \citet[Proposition 6.1]{catoni2012challenging} (note that $r_t$ here is a data-independent constant), there exists $\mu_0 \in \mathbb R$ such that when $X_i \overset{\mathrm{iid}}{\sim}  \mathcal{N}(\mu_0, \sigma^2)$,
\begin{equation}
     \Pr \left[  \htheta_t \ge \mu_0 + r_t  \right] \ge  \Pr \left[  \hmu_t \ge \mu_0 + r_t  \right], \text{or }  \Pr \left[  \htheta_t \le \mu_0 - r_t  \right] \ge  \Pr \left[  \hmu_t \le \mu_0 - r_t  \right].
\end{equation}
Without loss of generality suppose the latter holds. Surely $\mathcal{N}(\mu_0, \sigma^2)^{\otimes \mathbb {N}^+} \in \mathcal{Q}_{\sigma^2}$. We see that
\begin{equation}
 \frac{\alpha}{2}+\varepsilon \ge  \Pr \left[  \htheta_t \le \mu_0 - r_t  \right] \ge  \Pr \left[  \hmu_t \le \mu_0 - r_t  \right],
\end{equation}
indicating that
\begin{equation}
    r_t \ge \frac{\sigma}{\sqrt{t}}\Phi^{-1}(1-\alpha/2 - \varepsilon).
\end{equation}
This shows that
\begin{equation}
    w^{\mathcal{N}(\mu_0, \sigma^2)^{\otimes \mathbb {N}^+}, \varepsilon}_t \ge \frac{2\sigma}{\sqrt{t}}\Phi^{-1}(1-\alpha/2 - \varepsilon)
\end{equation}
holds for \emph{any} tail-symmetric CI, which clearly imlpies the minimax lower bound.
\end{proof}

\begin{proof}[Proof of \cref{prop:lil}]
By the law of the iterated logarithm,
\begin{equation}
    \limsup_{t \to \infty} \frac{|\hmu_t - \mu|}{\sigma}\sqrt{\frac{t}{2\log\log t}} = 1,\; \text{almost surely.}
\end{equation}
With probability at least $1-\alpha$, $\hmu_t, \mu \in \CI_t$ for every $t$, which implies that $|\CI_t|\ge |\hmu_t - \mu|$ for every $t$.
Hence, with probability at least $1-\alpha$,
\begin{equation}
    \limsup_{t \to \infty} \frac{|\CI_t|}{\sigma}\sqrt{\frac{t}{2\log\log t}} \ge 1.
\end{equation}
\end{proof}

\begin{lemma}\label{lem:comparison}
Let $\{ R_t^m \}$ and $\{ S_t^m \}$ be two families of nonnegative adapted processes indexed by $m \in \mathbb R$, among which $\{ R_t^\mu \}$ and $\{ S_t^\mu \}$ are supermartingales with $R_0^\mu = S_0^\mu = 1$. If almost surely for any $m \in \mathbb R$, $R_t^m \ge S_t^m$, then the $(1-\alpha)$-CSs
\begin{equation}
    \CI_{R, t} = \{ m : R_t^m \le 1/\alpha  \}, \quad \CI_{S, t} = \{ m : S_t^m \le 1/\alpha  \} 
\end{equation}
satisfy
\begin{equation}
    \CI_{R, t} \subseteq \CI_{S, t} \quad \text{a.s.}
\end{equation}
\end{lemma}

\begin{proof}[Proof of \cref{lem:comparison}]
First, by Ville's inequality (\cref{lem:vi}) we see that $\{ \CI_{R, t}\}$ and $\{ \CI_{R, t}\}$ are indeed confidence sequences for $\mu$. Then, almost surely, for all $m \in \CI_{R, t}$,
\begin{equation}
   S_t^m \le  R_t^m \le 1/\alpha,
\end{equation}
which implies that $m \in \CI_{S, t} $. Hence almost  surely $ \CI_{R, t} \subseteq \CI_{S, t}$.
\end{proof}

Before we prove \cref{thm:infvar-tight}, let us introduce three lemmas, which are also used to prove the infinite variance case in the recent follow-up work on \emph{robust} Catoni-style confidence sequences, \cite{wang2023huber}. Our proof of \cref{thm:infvar-tight} is also inspired by \cite[Proof of Theorem 5]{wang2023huber}, which in turns roughly follows the proof of \cref{thm:catoni-tight} in this paper.

\begin{lemma}[Lemma 8 of \cite{wang2023huber}, Appendix A]\label{lem:p-poly-zero}
     Let $p \in (1,2]$, $C > 0$, and $
    g(y) = y^p - y + C$.  Suppose there is a $c > 0$ such that $C = \left(\frac{c}{(1+c)^p} \right)^{1/(p-1)}$. Then $g((1+c)C)=0$.
\end{lemma}

\begin{lemma}[Lemma 9 of \cite{wang2023huber}, Appendix A]\label{lem:p-poly-zero-2}
    Let $p \in (1,2]$, $A,B,C>0$, and $
    g(x) = Ax^p - Bx + C$. Suppose there is a $c>0$ such that $ A^{1/(p-1)} B^{-p/(p-1)} C =  \left(\frac{c}{(1+c)^p} \right)^{1/(p-1)}$.
    Then $g$ has a positive zero $(1+c) B^{-1} C $.
\end{lemma}

\begin{lemma}[One-dimensional special case of  Lemma 7 of \cite{wang2021convergence}, Appendix A]\label{lem:neurips-expand-lemma} Let $p \in (1,2]$. For any real $x$ and $y$, $|x+y|^p \le |x|^p + 4|y|^p + py|x|^{p-1} \sg(x)$.
\end{lemma}

The first two lemmas are proved by direct substitution; the third by Taylor expansion. We refer the reader to the works cited for their proof. Now we are ready to prove \cref{thm:infvar-tight}.

\begin{proof}[Proof of \cref{thm:infvar-tight}] We define $f_{pt}(m)$ to be the random function
\begin{equation}
    f_{pt}(m) = \sum_{i=1}^t \phi_p(\lambda_i(X_i - m)).
\end{equation}
    Consider the process
    \begin{equation}
        {\Mcat_t}^{p}(m) = \prod_{i=1}^t \exp \left(\phi_p(\lambda_i (X_i  - m)) - \lambda_i(\mu - m) - \frac{\lambda^p_i}{p} (4 v_i + |\mu - m|^p) \right).
    \end{equation}
    Note that
    \begin{align}
        & \Exp\left[  \exp \left(\phi(\lambda_t (X_t  -m)) - \lambda_t(\mu - m) - \frac{\lambda^p_t}{p} (4 v_t + |\mu - m|^p) \right) \middle \vert \mathcal{F}_{t-1}  \right]
        \\
        = & \Exp\left[  \exp \left(\phi(\lambda_t (X_t  - m))  \right) \middle \vert \mathcal{F}_{t-1}  \right] \exp\left( - \lambda_t(\mu - m) - \tfrac{\lambda^p_t}{p} (4 v_t + |\mu - m|^p) \right)
        \\
        \le & \Exp[ 1 + \lambda_t (X_t-m) + \tfrac{1}{p}\lambda^p_t|X_t - m|^p  | \mathcal{F}_{t-1}] \exp\left( - \lambda_t(\mu - m) - \tfrac{\lambda^p_t}{p} (4 v_t + |\mu - m|^p) \right)
        \\
         & \text{(by \cref{lem:neurips-expand-lemma})}
        \\
        \le &  \frac{\Exp[ 1 + \lambda_t (X_t-m) + \tfrac{1}{p}\lambda^p_t ( |\mu - m|^p + 4|X_t - \mu|^p + p(X_t - \mu) |\mu - m|^{p-1} \sg(\mu - m) )  | \mathcal{F}_{t-1}] } {\exp\left(  \lambda_t(\mu - m) + \tfrac{\lambda^p_t}{p} (4 v_t + |\mu - m|^p) \right) }
        \\
        =  &  \frac{  1 + \lambda_t (\mu-m) + \tfrac{1}{p}\lambda^p_t ( |\mu - m|^p + 4\Exp[|X_t - \mu|^p | \mathcal{F}_{t-1}]  ) } {\exp\left(  \lambda_t(\mu - m) + \tfrac{\lambda^p_t}{p} (4 v_t + |\mu - m|^p) \right) }
        \\
        \le &  \frac{  1 + \lambda_t (\mu-m) + \tfrac{1}{p}\lambda^p_t ( |\mu - m|^p + 4v_t^2  ) } {\exp\left(  \lambda_t(\mu - m) + \tfrac{\lambda^p_t}{p} (4 v_t + |\mu - m|^p) \right) } \le 1.
    \end{align}
    $\{{\Mcat_t}^{p}(m)\}$ is thus a nonnegative supermartingale issued at 1. Note that this time, unlike the finite variance case, $\{{\Mcat_t}^{p}(\mu)\}$ does \emph{not} equal $\{{\Mcat_t}^{p}\}$. We thus have $\Exp{\Mcat_t}^{p}(m) \le 1$; that is,
    \begin{equation}
        \Exp \exp(f_{pt}(m)) \le \exp\left(  \sum_{i=1}^t \left( \lambda_i(\mu - m) + \frac{\lambda^p_i}{p} (4 v_i + |\mu - m|^p) \right) \right).
    \end{equation}
Define the function
\begin{equation}
    B_{pt}(m) = \sum_{i=1}^t \left( \lambda_i(\mu - m) + \frac{\lambda^p_i}{p} (4 v_i + |\mu - m|^p) \right) + \log(2/\varepsilon).
\end{equation}
By Markov's inequality,
\begin{equation}\label{eqn:markov-p}
    \forall m \in \mathbb R, \ \Pr[ f_{pt}(m) \le  B_{pt}(m) ] \ge 1- \varepsilon/2.
\end{equation}
Let us now consider the equation
\begin{equation}\label{eqn:zero-eqn-p-real}
    B_{pt}(m) = - \frac{\sum_{i=1}^t  v_i \lambda_i^p }{p} - \log(2/\alpha),
\end{equation}
which, by rearrangement, can be written as (each $\sum$ standing for $\sum_{i=1}^t$),
\begin{equation}\label{eqn:zero-eqn-p}
    \left( \sum \lambda_i^p/p \right) |m - \mu|^p -  \left( \sum \lambda_i \right)(m - \mu) + \left( \sum 5 \lambda_i^p v_i/ p + \log(2/\varepsilon) + \log(2/\alpha) \right) = 0.
\end{equation}
By \cref{lem:p-poly-zero-2}, it has a positive zero if there exists a $c>0$ such that
\begin{equation}
   \left( \sum \lambda_i^p/p \right)^{1/(p-1)} \left( \sum \lambda_i \right)^{-p/(p-1)} \left( \sum 5 \lambda_i^p v_i/ p + \log(2/\varepsilon) + \log(2/\alpha) \right) = \left( \frac{c}{(1+c)^p} \right)^{1/(p-1)}.
\end{equation}
By elementary calculus, the function $\frac{x}{(1+x)^p}$ ($x > 0$) takes its maximum $J = \frac{(p-1)^{p-1}}{p^p}$ at $x_p = \frac{1}{p-1}$. Therefore, due to the assumption \eqref{eqn:quadratic-condition-p}, the left hand side of the equation above is smaller than the maximum of the right hand, meaning that such $c$ does exist, and is smaller than $x_p = \frac{1}{p-1}$. So, the equation \eqref{eqn:zero-eqn-p}, i.e.\ the equation \eqref{eqn:zero-eqn-p-real}, has a positive zero
\begin{align}
    &\pi_t = \mu + (1+c) \left( \sum \lambda_i \right)^{-1} \left( \sum 5 \lambda_i^p v_i/ p + \log(2/\varepsilon) + \log(2/\alpha) \right)
    \\
    \le & \mu + \frac{p}{p-1} \frac{  \sum 5 \lambda_i^p v_i/ p + \log(2/\varepsilon) + \log(2/\alpha)  }{\sum \lambda_i }  =\mu +  \frac{  \sum 5 \lambda_i^p v_i + p\log(2/\varepsilon) + p\log(2/\alpha)  }{\sum \lambda_i}.
\end{align}
Let us put $m = \pi_t$ into \eqref{eqn:markov-p}, obtaining
\begin{equation}
    \Pr\left[ f_{pt}(\pi_t) \le   - \frac{\sum_{i=1}^t  v_i \lambda_i^p }{p} - \log(2/\alpha)   \right] \ge 1 - \varepsilon/2,
\end{equation}
which in turns implies, with probability at least $1-\varepsilon/2$,
\begin{equation}
     \max({\CICT_t}^p) \le \pi_t \le \mu +  \frac{  \sum 5 \lambda_i^p v_i + p\log(2/\varepsilon) + p\log(2/\alpha)  }{\sum \lambda_i}.
\end{equation}
The other side of the concentration similarly follows. So
\begin{equation}
    \Pr\left[ \max({\CICT_t}^p) - \min({\CICT_t}^p) \le \frac{  \sum 10 \lambda_i^p v_i + 2p\log(2/\varepsilon) + 2p\log(2/\alpha)  }{\sum \lambda_i} \right],
\end{equation}
concluding the proof.\end{proof}

\section{Heteroscedastic Dubins-Savage and self-normalized CS}\label{sec:het-sce}
Suppose in this section that instead of \cref{ass:var}, the following assumption holds.
\begin{assumption}\label{ass:varhet}
    The process is conditionally square-integrable with an upper bound, known \emph{a priori}, on the conditional variance:
    \begin{equation}\label{eqn:var-bound-het-reprise}
        \forall t\in\mathbb N^+, \quad \Exp[(X_t -\mu)^2  \mid \mathcal{F}_{t-1}] \le \sigma^2_t,
    \end{equation}
    where $\{\sigma_t\}_{t \in \mathbb N^+}$ is a predictable, nonnegative process.
\end{assumption}


We can easily generalize \cref{thm:cids} as follow\footnote{We develop this section without any proof since all theorems here can be easily deduced by modifying previously proved counterparts, changing every application of $\Exp[(X_t -\mu)^2  \vert \mathcal{F}_{t-1}] \le \sigma^2$ into $\Exp[(X_t -\mu)^2  \vert \mathcal{F}_{t-1}] \le \sigma^2_t$.}.
\begin{theorem}[Dubins-Savage confidence sequence]\label{thm:cids-het} The following intervals $\{\CIDS_t\}$, under \cref{ass:mean} and \cref{ass:varhet}, form a $(1-\alpha)$-confidence sequence of $\mu$:
\begin{equation}
    \CIDShet_t = \left[ \frac{\sum \lambda_i X_i \pm \left(2/\alpha - 1 + \sum \lambda_i^2 \sigma_i^2 \right)  }{\sum \lambda_i } \right].
\end{equation}
\end{theorem}

If we consider the simple case of deterministic polynomial growth $\sigma_t^2  = \Otheta(t^{\gamma})$, we see that matching the growth of $\{ \sigma_t \}$ with $\{ \lambda_t \}$ via setting $\lambda_t  = \Otheta(t^{-(\gamma + 1)/2})$ can result in the best possible asymptotic shrinkage rate
\begin{equation}
    |\CIDShet_t| = \Otheta(t^{(\gamma-1)/2}),
\end{equation}
which is indeed a ``shrinkage" if $\gamma < 1$, cf.\ \cref{cor:infvar-tight-2}.

Similarly, \cref{thm:cisn} can be reformulated in the following way.
\begin{theorem}[Self-normalized confidence sequence] 
We define the intervals $\aCISNphet$ to be
\begin{gather}
    \aCISNphet_t = \left[ \frac{ \left( U_t^{+} \right) \pm \sqrt{  \left( U_t^{+}  \right) ^2 -  \frac{2\sum \lambda_i^2}{3}  \left( \log(2/\alpha) - \sum \lambda_i X_i  + \frac{\sum \lambda_i^2 X_i^2 +  2\sum \lambda_i^2 \sigma_i^2  }{6} \right)  }    }{ \frac{\sum \lambda_i^2}{3} }  \right];
    \\
    \aCISNnhet_t = \left[ \frac{ \left( U_t^{-} \right) \pm \sqrt{  \left( U_t^{-}  \right) ^2 -  \frac{2\sum \lambda_i^2}{3}  \left( \log(2/\alpha) + \sum \lambda_i X_i  + \frac{ \sum \lambda_i^2 X_i^2 +  2\sum \lambda_i^2 \sigma_i^2}{6} \right)  }    }{ \frac{\sum \lambda_i^2}{3} }  \right].
\end{gather}
($U^\pm _t$ are defined back in \eqref{eqn:upm}.) Then, setting $\CISNhet_t = \mathbb R \setminus (\aCISNphet_t \cup \aCISNnhet_t)$, we have that $\{ \CISNhet_t  \}$ forms a $1-\alpha$ confidence sequence for $\mu$, under \cref{ass:mean} and \cref{ass:varhet}.
\end{theorem}





\end{document}